\documentclass[11pt]{amsart}
\usepackage{amsmath,amsthm,amssymb,latexsym,color}
\usepackage{hyperref}
\usepackage{lipsum}

\date{}

\newtheorem{theorem}{Theorem}[section]
\newtheorem*{theorem*}{Theorem}
\newtheorem*{theoremA}{Theorem~A}
\newtheorem*{theoremB}{Theorem~B}

\newtheorem{lemma}[theorem]{Lemma}
\newtheorem{cor}[theorem]{Corollary}

\newtheorem{prop}[theorem]{Proposition}

\theoremstyle{definition}
\newtheorem{Remark}[theorem]{Remark}

\theoremstyle{plain}

\newcommand{\N}{\mathbb{N}}
\newcommand{\Z}{\mathbb{Z}}
\newcommand{\R}{\mathbb{R}}

\newcommand{\Exp}{\mathbb{E}}

\newcommand{\Event}{\mathcal{E}}
\def\Prob{{\mathbb P}}
\def\cf{{\mathcal L}}

\newcommand{\supp}{{\rm supp\,}}

\def\Col{{\rm col}}
\def\Row{{\rm row}}
\def\comp{{\rm Comp}}
\def\incomp{{\rm Incomp}}
\def\Net{{\mathcal N}}
\def\spn{{\rm span\,}}
\def\dist{{\rm dist\,}}

\def\thres{{\mathcal T}}
\def\spv{{\bf Y}}
\def\LCD{{\rm LCD}}

\title{Singularity of random Bernoulli matrices}
\author{Konstantin Tikhomirov}
\address{School of Mathematics,
Georgia Institute of Technology}
\email{ktikhomirov6@gatech.edu}

\begin{document}

\begin{abstract}
For each $n$, let $M_n$ be an $n\times n$ random matrix with independent $\pm 1$ entries.
We show that $\Prob\{\mbox{$M_n$ is singular}\}=(1/2+o_n(1))^n$, which settles an old problem.
Some generalizations are considered. 
\end{abstract}

\maketitle

\section{Introduction}

Let $X_1,X_2,\dots,X_n$ be independent vectors, each $X_i$ uniformly distributed on vertices of the discrete cube $\{-1,1\}^n$.
What is the probability that $X_1,\dots,X_n$ are linearly independent?

The question has attracted considerable attention in literature.
It can be equivalently restated as a question about singularity of an $n\times n$ matrix $M_n$ with independent
$\pm 1$ entries.
J.\ Koml\'os \cite{Komlos} showed that $\Prob\{\mbox{$M_n$ is singular}\}=o_n(1)$.
Much later, the bound $\Prob\{\mbox{$M_n$ is singular}\}\leq 0.999^n$ was obtained
by J. Kahn, J. Koml\'{o}s\ and\ E. Szemer\'{e}di
in \cite{KKS}.
The upper bound was sequentially improved to $0.939^n$ in \cite{TV disc1} and $(3/4+o_n(1))^n$
in \cite{TV disc2} by T. Tao and V.Vu, and to $(1/\sqrt{2}+o_n(1))^n$ by J. Bourgain, V. Vu and P. Wood in \cite{BVW}.

It has been conjectured that
\begin{equation}\label{eq: conj}
\Prob\{\mbox{$M_n$ is singular}\}=\bigg(\frac12+o_n(1)\bigg)^n
\end{equation}
(see, for example, \cite[Conjecture~1.1]{BVW}, \cite[Conjecture~7.1]{Vu survey},
\cite[Conjecture~2.1]{Vu ICM 2014} as well as some stronger conjectures in \cite{Arratia}).
In this paper, we confirm the conjecture and, moreover,
provide quantitative small ball probability estimates for the smallest singular value of $M_n$.
We extend our analysis to random matrices with Bernoulli($p$)
independent entries.
Let $1_n$ denote the $n$--dimensional vector of all ones.
The main result of this paper can be formulated as follows.
\begin{theoremA}
For every $p\in(0,1/2]$ and $\varepsilon>0$ there are $n_{\text{\tiny{p,$\varepsilon$}}},C_{\text{\tiny{p,$\varepsilon$}}}>0$
depending only on $p$ and $\varepsilon$ with the following property.
Let $n\geq n_{\text{\tiny{p,$\varepsilon$}}}$, and let $B_n(p)$ be $n\times n$ random matrix with independent
entries $b_{ij}$, such that $\Prob\{b_{ij}=1\}=p$ and $\Prob\{b_{ij}=0\}=1-p$.
Then for any $s\in[-1,0]$
$$\Prob\big\{s_{\min}(B_n(p)+s\,1_n1_n^\top)\leq t/\sqrt{n}\big\}
\leq\big(1-p+\varepsilon\big)^n + C_{\text{\tiny{p,$\varepsilon$}}}\,t,\quad t>0.$$
%In particular,
%$$\Prob\big\{\mbox{$B_n(p)$ is singular}\big\}\leq \big(1-p+\varepsilon\big)^n.$$
\end{theoremA}
It is easy to see that the probability that the first column of $B_n(p)$ is equal to zero, is $(1-p)^n$.
Thus, the theorem implies that, for a fixed $p\in(0,1/2]$,
$$
\Prob\big\{\mbox{$B_n(p)$ is singular}\big\}=\big(1-p+o_n(1)\big)^n,
$$
and further, when applied with $p=1/2$ and $s=-1/2$, gives \eqref{eq: conj}.

\section{Proof strategy}\label{s: strategy}

The proof of upper bounds on the probability of singularity of random discrete matrices (i.e.\ matrices with entries
taking a finite number of values)
developed in work \cite{KKS} and later in \cite{TV disc1, TV disc2, BVW},
uses, as a starting point, the relation
\begin{align*}
\Prob\big\{\mbox{the matrix with columns $X_1,\dots,X_n$ is singular}\big\}
&=e^{o_n(n)}\,\Prob\big\{\mbox{the matrix has rank $n-1$}\big\}\\
&=e^{o_n(n)}\sum\limits_{V}\Prob(A_V),
\end{align*}
which holds under rather broad assumptions on the distributions of the discrete random vectors $X_1,\dots,X_n$
\cite{BVW}.
Here, the summation is taken over (finitely many) hyperplanes $V$ %passing through the origin
such that
the probability of $A_V$ --- the event that $X_1,\dots,X_n$ span $V$ --- is non-zero.
The set of the hyperplanes $V$ is then partitioned according to the value of the {\it combinatorial dimension}
which is defined as the number $d(V)\in \frac{1}{n}\Z$ such that $\max\limits_i\Prob\{X_i\in V\}\in\big(C^{-d(V)-1/n},C^{-d(V)}\big]$,
where $C$ is some constant depending on the distribution of $X_i$'s.
The sum of probabilities corresponding to a given combinatorial dimension is estimated
in terms of probabilities $\Prob\{Y_i\in V\}$ for specially constructed random vectors $Y_i$.
For some discrete distributions, in particular, for matrices with i.i.d.\ entries with the probability mass function
$$f(m)=\begin{cases}\frac{1}{4},&\mbox{ if $m=\pm 1$};\\ \frac{1}{2},&\mbox{ if $m=0$,}\end{cases}$$
upper bounds for the singularity obtained using the strategy are asymptotically sharp as was shown in \cite{BVW}.

Methods %which, in addition to the qualitative property of being invertible, provide
providing strong quantitative information on the smallest
singular value of a random matrix were proposed in papers \cite{Rudelson ann math,TV ann math}.
As a further development, the work \cite{RV adv}
established small ball probability estimates on $s_{\min}$ of any $n\times n$ matrix $A_n$ with i.i.d normalized subgaussian entries
of the form $\Prob\{s_{\min}(A_n)\leq t/\sqrt{n}\}\leq c^{n}+Ct$, $t>0$, where $C>0$ and $c\in(0,1)$
depend only on the subgaussian moment. Thus, \cite{RV adv} recovered the result of \cite{KKS},
possibly with a worse constant. The key notion of \cite{RV adv}
is {\it the essential least common denominator} (LCD) which measures ``unstructuredness'' of a fixed vector $(x_1,\dots,x_n)$
and is defined as the smallest $\lambda$ such that the distance from $\lambda x$
to the integer lattice $\Z^n$ does not exceed $\min(c'\lambda\|x\|_2, c\sqrt{n})$.
LCD can be used to characterize anticoncentration properties of random sums $\sum_{i}a_{ij}x_i$
(and in that respect the approach of \cite{RV adv} is related to the earlier paper \cite{TV ann math}
where the anticoncentration properties of discrete random sums were connected with existence of
{\it generalized arithmetic progressions} containing almost all of $\{x_1,\dots,x_n\}$).
It was proved in \cite{RV adv} that for any unit vector $x$,
$\Prob\big\{\big|\sum_{i}a_{ij}x_i\big|\leq t\big\}\leq C t+\frac{C}{\LCD(x)}+e^{-cn}$
for any $t>0$ (see also \cite{RV rect}).
This relation, combined with the assertion that the LCD of a random unit vector normal to the linear span of the first $n-1$ columns of $A_n$
is exponential in $n$, already implies that $A_n$ is singular with probability at most $e^{-cn}$.
Moreover, an efficient averaging procedure (which we recall below) used in \cite{RV adv}
allows to obtain strong quantitative bounds on $s_{\min}(A_n)$.
The LCD of the random unit normal is estimated 
with help of an elaborate $\varepsilon$--net argument.

\medskip

The approach that we use in this paper is partially based on the methods used in \cite{RV adv}
(and in \cite{LPRT}), while the principal difference lies in estimating anticoncentration properties of random sums.
%, while the central part of the argument which measures anticoncentration
%properties of scalar products with the normal vectors, is different.
The starting point is the relation (taken from \cite{RV adv})
\begin{equation*}
\begin{split}
\Prob\big\{s_{\min}(A_n)\leq t/\sqrt{n}\big\}
&\leq \Prob\big\{\|A_nx\|_2\leq t/\sqrt{n}\mbox{ for some }x\in \comp_n(\delta,\nu)\big\}\\
&\hspace{1cm}+\Prob\big\{\|A_nx\|_2\leq t/\sqrt{n}\mbox{ for some }x\in \incomp_n(\delta,\nu)\big\}\\
&\leq\Prob\big\{\|A_nx\|_2\leq t/\sqrt{n}\mbox{ for some }x\in \comp_n(\delta,\nu)\big\}\\
&\hspace{1cm}+\frac{1}{\delta}\Prob\big\{|\langle\Col_n(A_n),Y_n\rangle|\leq t/\nu\}\big\},
\end{split}
\end{equation*}
valid for any $n\times n$ random matrix $A_n$ with the distribution invariant under permutations of columns.
Here, $Y_n$ is a random unit vector orthogonal to the linear span of $\Col_1(A_n),\dots$, $\Col_{n-1}(A_n)$;
$\comp_n(\delta,\nu)$ is the set of {\it compressible}
unit vectors defined as those with the Euclidean distance
at most $\nu$ to the set of $\delta n$--sparse vectors;
$\incomp_n(\delta,\nu)=S^{n-1}\setminus \comp_n(\delta,\nu)$ is the set of {\it incompressible} vectors.
In the above formula, $\delta,\nu\in(0,1]$ can be arbitrary, although for our proof we take both parameters
small (depending on the choice of $\varepsilon$ in the statement of our main result).

The first summand in the rightmost expression --- the small ball probability for $\inf\limits_{x\in\comp_n}\|Ax\|_2$
--- can be bounded with help of an argument which is completely standard by now.
For Reader's convenience, we provide the estimate together with a complete proof in Preliminaries.

The second term --- $\Prob\big\{|\langle\Col_n(A_n),Y_n\rangle|\leq t/\nu\big\}$
--- crucially depends on the structure of the random normal $Y_n$.
In \cite{RV adv}, the authors provided an explicit characterization of ``unstructured'' vectors
in terms of the LCD. In contrast, in our approach we make no attempt to obtain a geometric
description of vectors with good anticoncentration properties.
For each unit vector $x$ and a parameter $L$, we introduce the {\it threshold} $\thres_p(x,L)$
which is defined as the supremum of all
$t\in(0,1]$ such that
$\cf\big(\sum_{i=1}^n b_i x_i,t\big)> Lt$,
where, $b_1,\dots,b_n$ are independent Bernoulli($p$) random variables.
Here, $\cf(\cdot,\cdot)$ denotes the L\'evy concentration function, defined as
$\cf(Z,t):=\sup_{\lambda\in\R}\Prob\{|Z-\lambda|\leq t\}$, $t\geq 0$, for any real valued random variable $Z$.
The threshold can be viewed as a lower bound of the range of $t$'s
for which corresponding random linear combination admits ``good'' anticoncentration estimates.
Thus, to show that $B_n(p)+s1_n 1_n^\top$ is singular with probability $(1-p+o_n(1))^n$,
it is sufficient to check that the threshold of the random normal $Y_n$ is at most $(1-p+o_n(1))^{n}$
with probability at least $1-(1-p+o_n(1))^n$.
Note that this approach can be related to the {\it inverse} Littlewood--Offord theory
started in \cite{TV ann math}, although here we are only interested in estimating from above the ``size'' of the
set of potential normal vectors with large thresholds, rather than giving an explicit description of this set
(in that respect, our strategy can be related to theorems in \cite[Section~3]{TV circular}, however, the actual proofs
are very different).

To estimate the threshold, we apply a procedure which can be called ``inversion of randomness'',
and which we briefly describe below.
We would like to make the description as non-technical as possible, and for this reason omit any
discussion of the choice of parameters and other issues of secondary importance.
Take any $T$ with $T^{-1}\ll (1-p+o_n(1))^{-n}$, and let $D_T$ be the set of all $(\delta,\nu)$--incompressible
unit vectors with the threshold falling into the interval $[T,2T)$.
In order to show that the probability of the event $\{Y_n\in D_T\}$ is close to zero,
we construct a discrete approximation $\Net_T$ of $D_T$,
which is a subset of elements of an $n$--dimensional lattice
having the threshold of order $T$, and coordinates in a certain range.
We then show that the event $\{Y_n\in D_T\}$ is contained in
$$\Event_{\Net_T}:=\big\{\mbox{There is a vector }x\in \Net_T\mbox{ ``almost orthogonal'' to $\Col_1,\dots,\Col_{n-1}$}\big\},$$
where ``almost orthogonal'' should be understood in a specific sense which we prefer not to discuss here.
This implies
$$
\Prob\{Y_n\in D_T\}\leq \Prob(\Event_{\Net_T})\leq |\Net_T|\,\max\limits_{x\in\Net_T}
\Prob\big\{\mbox{$x$ is ``almost orthogonal'' to $\Col_1,\dots,\Col_{n-1}$}\big\},
$$
and the proof is reduced to efficiently bounding from above the cardinality of the discretization $\Net_T$.
The ``inversion of randomness'' is used to solve the problem.
We consider a random vector $\xi$ uniformly distributed on a subset of the lattice
(whose cardinality is much easier to compute) containing $\Net_T$,
and show that with probability {\it superexponentially} close to one, the threshold of $\xi$
is much less than $T$, so that $\xi\notin \Net_T$. This allows to bound $|\Net_T|$
in terms of the cardinality of the range of $\xi$, times the factor $e^{-\omega(n)}$.
Thus, instead of studying anticoncentration of random sums with fixed coefficients
satisfying certain structural assumptions, we consider typical anticoncentration properties
of sums with {\it random} coefficients $\xi_i$.
It will be convenient to work with the expression
$$
\cf_b \Big(\sum\limits_{i=1}^n b_i\xi_i,t\Big)
:=\sup\limits_{\lambda\in\R}\sum\limits_{(v_j)_{j=1}^n\in\{0,1\}^n}
p^{\sum_j v_j} (1-p)^{n-\sum_j v_j}{\bf 1}_{[-t,t]}\big(\lambda+v_1\xi_1+\dots+v_n\xi_n\big),
$$
which is interpreted as the L\'evy concentration function with respect to
the randomness of the vector $b=(b_1,\dots,b_n)$ of independent Bernoulli($p$) components.

Let us state, as an illustration, a corollary of the main technical result of this paper, Theorem~\ref{th: averaging},
which deals with rescaled vectors distributed on the integer lattice $\Z^n$:
\begin{theoremB}
Let $\delta\in(0,1]$, $p\in(0,1/2]$, $\varepsilon\in(0,p)$, $M\geq 1$.
There exist $n_{\text{\tiny B}}=n_{\text{\tiny B}}(\delta,\varepsilon,p,M)\geq 1$
depending on $\delta,\varepsilon,p,M$ and $L_{\text{\tiny B}}=
L_{\text{\tiny B}}(\delta,\varepsilon,p)>0$ depending {\bf only} on $\delta,\varepsilon,p$
(and not on $M$)
with the following property. Take $n\geq n_{\text{\tiny B}}$, $1\leq N\leq (1-p+\varepsilon)^{-n}$,
and let
$$\mathcal A:=\{-2N,\dots,-N-1,N+1,\dots,2N\}^{\lfloor \delta n\rfloor}\times
\{-N,-N+1,\dots,N\}^{n-\lfloor \delta n\rfloor}.$$
Further, assume that a random vector $\xi=(\xi_1,\dots,\xi_n)$ is uniform on $\mathcal A$.
Then
$$\Prob_\xi\big\{\cf_b \big(b_1\xi_1+\dots+b_n\xi_n,\sqrt{n}\big)
> L_{\text{\tiny B}} N^{-1}
\big\}\leq e^{-M\,n}.$$
Here, $\cf_b(\cdot,\cdot)$ denotes the L\'evy concentration function with respect to $b=(b_1,\dots,b_n)$,
a random vector with independent Bernoulli($p$) components. 
\end{theoremB}
The crucial point of this theorem is that $L_{\text{\tiny B}}$ does not depend on $M$.
Essentially, this means that the probability can be made superexponentially small in $n$ as $n$ grows,
while $L_{\text{\tiny B}}$ stays constant.
Because of the ``inversion of randomness'', a statement of this kind is translated into bounds
for the cardinality of the discretization of the sets of vectors $D_T$ with large thresholds
considered above.

\section{Preliminaries}

Denote by $\|\cdot\|_q$ the standard $\ell_q$--norm, so that
$$\big\|(x_1,x_2,\dots)\big\|_q=\bigg(\sum\limits_{i}|x_i|^q\bigg)^{1/q},\quad 1\leq q<\infty;\quad
\mbox{and}\quad \big\|(x_1,x_2,\dots)\big\|_\infty=\max\limits_i |x_i|.$$
In particular, by $\ell_1(\Z)$ we denote the space of all functions $g:\Z\to\R$ with $\sum_i |g(i)|<\infty$.
We will say that a mapping $g:\Z\to\R$ is {\it $L$--Lipschitz} for some $L>0$ if $|g(t)-g(t+1)|\leq L$
for all $t\in\Z$.

The unit Euclidean sphere in $\R^n$ will be denoted by $S^{n-1}$.
The support of a vector $y=(y_1,\dots,y_n)\in\R^n$ is $\supp y:=|\{i\leq n:\;y_i\neq 0\}|$.
The $n$--dimensional vector of all ones is denoted by $1_n$.
For an $n\times n$ matrix $A$, $\Col_i(A)$ and $\Row_i(A)$ are its columns and rows, respectively,
and $\|A\|$ is the spectral norm of $A$. The smallest singular value of $A$ is denoted by $s_{\min}(A)$.
We will rely on the standard representation $s_{\min}(A)=\min\limits_{x\in S^{n-1}}\|Ax\|_2$.

The indicator of a subset of $\R$ or an event $S$ is denoted by ${\bf 1}_S$.
For any positive integer $m$, $[m]$ denotes the integer interval $\{1,2,\dots,m\}$.
Further, for any two subsets $I,J\subset\Z$, we write $I<J$ if $i<j$ for all $i\in I$ and $j\in J$.
The {\it Minkowski sum} of two subsets $A,B$ of $\R^m$ is defined as the set of all vectors of the form $a+b$,
where $a\in A$ and $b\in B$.
For a real number $r$, by $\lfloor r\rfloor$ we denote the largest integer less than or equal to $r$, and by $\lceil r\rceil$,
the smallest integer greater than or equal to $r$.

Everywhere in this paper, $B_n(p)$ is the matrix with i.i.d.\ Bernoulli($p$) entries, i.e.\ random variables taking value $1$ with probability $p$
and $0$ with probability $1-p$.
Further, by $B_n^1(p)$
we denote the $(n-1)\times n$ matrix obtained from $B_n(p)$ by removing the last row.

The {\it L\'evy concentration function} $\cf(\xi,\cdot)$ of a random variable $\xi$ is defined by
$$\cf\big(\xi,t\big):=\sup\limits_{\lambda\in\R}\Prob\big\{|\xi-\lambda|\leq t\big\},\quad t\geq 0.$$
We will need the following classical inequality:
\begin{lemma}[L\'evy--Kolmogorov--Rogozin, \cite{Rogozin}]\label{l: lkr}
Let $\xi_1,\dots,\xi_m$ be independent real valued random variables. Then for any real numbers $r_1,\dots,r_m>0$
and $r\geq \max_{i\leq m} r_i$,
$$
\cf\Big(\sum_{i=1}^m \xi_i,r\Big)\leq \frac{C_{\text{\tiny\ref{l: lkr}}} r}{\sqrt{\sum_{i=1}^m (1-\cf(\xi_i,r_i))r_i^2}}.
$$
Here, $C_{\text{\tiny\ref{l: lkr}}}>0$ is a universal constant.
\end{lemma}

We recall some definitions from \cite{RV adv}.
Given $\delta\in(0,1]$ and $\nu\in(0,1]$, denote by $\comp_n(\delta,\nu)$ the set of all unit vectors $x\in\R^n$
such that there is $y=y(x)\in\R^n$ with $|\supp y|\leq \delta n$ and $\|x-y\|_2\leq \nu$
(in \cite{RV adv}, such vectors are called {\it compressible}).
Further, we define the complementary set of {\it incompressible} vectors 
$\incomp_n(\delta,\nu):=S^{n-1}\setminus \comp_n(\delta,\nu)$.
We note that a similar partition of the unit vectors was used earlier in \cite{LPRT}.

Following an approach developed in \cite{RV adv}, we can write
for any random matrix $A_n$ with the distribution invariant under permutations of columns
\begin{equation}\label{eq: comp-incomp}
\begin{split}
\Prob\big\{s_{\min}(A_n)\leq t/\sqrt{n}\big\}
&\leq \Prob\big\{\|A_nx\|_2\leq t/\sqrt{n}\mbox{ for some }x\in \comp_n(\delta,\nu)\big\}\\
&\hspace{1cm}+\Prob\big\{\|A_nx\|_2\leq t/\sqrt{n}\mbox{ for some }x\in \incomp_n(\delta,\nu)\big\}\\
&\leq\Prob\big\{\|A_nx\|_2\leq t/\sqrt{n}\mbox{ for some }x\in \comp_n(\delta,\nu)\big\}\\
&\hspace{1cm}+\frac{1}{\delta}\Prob\big\{|\langle\Col_n(A_n),Y_n\rangle|\leq t/\nu\}\big\},
\end{split}
\end{equation}
where $\delta,\nu$ are arbitrary numbers in $(0,1)$ (see \cite[formula~(3.2) and Lemma~3.5]{RV adv}),
and $Y_n$ is a random unit vector orthogonal to the first $n-1$ columns of $A_n$.
A satisfactory estimate for the first term for sufficiently small $\delta$ and $\nu$ can be obtained as a simple compilation
of known results (see Proposition~\ref{l: compress} below).
The following is a version of the {\it tensorization lemma} from \cite{RV adv}.
\begin{lemma}\label{l: tensorization}
Let $\xi_1,\dots,\xi_m$ be independent random variables.
\begin{itemize}

\item[(1)] Assume that for some $\varepsilon_0>0$, $K>0$ and all $\varepsilon\geq \varepsilon_0$ and $k\leq m$ we have
$$\Prob\big\{|\xi_k|\leq \varepsilon\}\leq K\varepsilon.$$
Then for each $\varepsilon\geq \varepsilon_0$,
$$\Prob\big\{\|(\xi_1,\xi_2,\dots,\xi_m)\|_2\leq \varepsilon\sqrt{m}\big\}\leq (C_{\text{\tiny\ref{l: tensorization}}}\,K\,\varepsilon)^m,$$
where $C_{\text{\tiny\ref{l: tensorization}}}>0$ is a universal constant.

\item[(2)] Assume that for some $\eta>0$, $\tau>0$ and all $k\leq m$
we have $\Prob\big\{|\xi_k|\leq \eta\}\leq \tau$.
Then for every $\varepsilon\in(0,1]$,
$$\Prob\big\{\|(\xi_1,\xi_2,\dots,\xi_m)\|_2\leq \eta\sqrt{\varepsilon m}\big\}\leq \bigg(\frac{e}{\varepsilon}\bigg)^{\varepsilon m}
\tau^{m-\varepsilon m}.$$
\end{itemize}
\end{lemma}
\begin{Remark}
The second assertion of the lemma follows immediately by noting that the condition
$\|(\xi_1,\xi_2,\dots,\xi_m)\|_2\leq \eta\sqrt{\varepsilon m}$ implies that $|\{i\leq m:\,
|\xi_i|> \eta\}|\leq \varepsilon m$.
For a proof of the first assertion, see \cite{RV adv}.
\end{Remark}

Further, we recall a standard estimate for the spectral norm of random matrices with i.i.d.\ subgaussian entries
(for a proof, see, for example, \cite[Theorem 5.39]{V12}).
\begin{lemma}\label{l: sp norm}
For any $M,L\geq 1$ there is $C_{\text{\tiny M,L}}>0$ depending only on $M$ and $L$ with the following property.
Let $n\geq 1$ and let $A$ be an $n\times n$ random matrix with i.i.d.\ entries $a_{ij}$ of zero mean,
and such that $(\Exp|a_{ij}|^q)^{1/q}\leq M\sqrt{q}$ for all $q\geq 1$. Then with probability at least $1-\exp(-Ln)$
we have $\|A\|\leq C_{\text{\tiny M,L}}\sqrt{n}$.
\end{lemma}

The following is an easy consequence of Lemma~\ref{l: tensorization}:
\begin{lemma}\label{l: aux single vector}
For any $p\in(0,1/2]$ there is
$\gamma_{\text{\tiny\ref{l: aux single vector}}}>0$ which may only depend on $p$,
such that for every $\varepsilon\in(0,1]$,
$n\geq 2$ and arbitrary $s\in\R$ and $x\in S^{n-1}$,
$$\Prob\big\{\big\|(B_n^1(p)+s\,1_{n-1}1_n^\top) x\big\|_2\leq \gamma_{\text{\tiny\ref{l: aux single vector}}}\sqrt{\varepsilon n}\big\}
\leq \bigg(\frac{e}{\varepsilon}\bigg)^{\varepsilon (n-1)}
(1-p)^{(n-1)(1-\varepsilon)}.$$
\end{lemma}
\begin{proof}
Let $b_1,\dots,b_n$ be i.i.d.\ Bernoulli($p$) random variables.
It is not difficult to check that
\begin{equation}\label{eq: aux 09823982-503}
\cf\Big(\sum_{i=1}^n b_i x_i, r\Big)\leq 1-p
\end{equation}
for some $r>0$ which may only depend on $p$.
For a proof of this fact, one may consider two possibilities: first when the vector $x$
has a ``large'' $\ell_\infty$--norm, in which case the assertion follows by conditioning on all $b_i$'s
except the one corresponding to the largest component of $x$, and, second, when the vector $x$ has a ``small''
$\ell_\infty$--norm in which case, by the Central Limit Theorem, the random linear combination is approximately normally distributed,
see, for example, \cite[Lemma~2.1]{CT}.

Applying the second assertion of the Tensorization Lemma to \eqref{eq: aux 09823982-503},
we get the statement.
\end{proof}

By combining Lemma~\ref{l: aux single vector} with an $\varepsilon$-net argument, we obtain
a small ball probability estimate for compressible vectors.
The only difference from a standard argument here is due to the fact that for $s\neq -p$,
the matrix $B_n^1(p)+s\,1_{n-1}1_n^\top$ has typical spectral norm of order $\Theta((s+p)n)$
rather than $\Theta(\sqrt{n})$ in the simplest setting of a centered random matrix with normalized independent entries.
The net therefore has to be made ``denser'' in the direction $1_n$.

\begin{prop}\label{l: compress}
For any $\varepsilon\in(0,1]$ and $p\in(0,1/2]$ there are $n_{\text{\tiny\ref{l: compress}}}\in\N$,
$\gamma_{\text{\tiny\ref{l: compress}}}>0$ and
$\delta_{\text{\tiny\ref{l: compress}}},\nu_{\text{\tiny\ref{l: compress}}}\in(0,1)$ depending only on $\varepsilon$ and $p$
such that for $n\geq n_{\text{\tiny\ref{l: compress}}}$ and arbitrary $s\in\R$,
$$\Prob\big\{\big\|(B_n^1(p)+s\,1_{n-1}1_n^\top) x\big\|_2\leq \gamma_{\text{\tiny\ref{l: compress}}}\sqrt{n}
\mbox{ for some }x\in
\comp_n(\delta_{\text{\tiny\ref{l: compress}}},\nu_{\text{\tiny\ref{l: compress}}})\big\}
\leq \big(1-p+\varepsilon\big)^{n}.$$
\end{prop}
\begin{proof}
Choose any $\varepsilon\in(0,1]$ and $p\in(0,1/2]$, and fix $s\in\R$.
It will be convenient to work with parameter $\widetilde s:=s+p$. Without loss of generality,
we can assume that $\widetilde s\neq 0$.
By Lemma~\ref{l: sp norm}, there is $L>0$ which may only depend on $p$ such that for every $n\geq 2$
the event
$$\Event:=\big\{\|B_n^1(p)-p\,1_{n-1}1_n^\top \|\leq L\sqrt{n}\big\}$$
has probability at least $1-2^{-n}$.

Given an $\widetilde \varepsilon\in(0,1]$ (which will be chosen later),
define
$$\delta:=\widetilde \varepsilon;\quad 
\gamma:=\gamma_{\text{\tiny\ref{l: aux single vector}}}\sqrt{\widetilde\varepsilon};\quad
\nu:=\frac{\gamma}{32L}.$$
We shall partition the set $\comp_n(\delta,\nu)$
into subsets $S_\ell$ of the form
$$S_\ell:=\comp_n(\delta,\nu)\cap \Big\{x\in\R^n:\;\sum\nolimits_{i=1}^n x_i \in
\Big[\frac{\gamma\ell}{4|\widetilde s|},\frac{\gamma(\ell+1)}{4|\widetilde s|}\Big)\Big\},
\quad \ell\in\Z.$$

First, we observe that a standard volumetric argument, together with the definition
of compressible vectors, implies that for any $\ell\in\Z$ the set $S_\ell$
admits a Euclidean $\big(\frac{\gamma}{16 L}+2\nu\big)$--net $\Net_\ell\subset S_\ell$
of cardinality at most ${n\choose{\lfloor \delta n\rfloor}}\big(\frac{C'L}
{\gamma}\big)^{\lfloor \delta n\rfloor}$,
for some universal constant $C'>0$.
By the definition of $\Net_\ell$ and $S_\ell$, for any $x\in S_\ell$ there is $y\in \Net_\ell$
such that $\|x-y\|_2\leq \big(\frac{\gamma}{16 L}+2\nu\big)=\frac{\gamma}{8 L}$ and
$\big|\sum_{i=1}^n (x_i-y_i)\big|\leq \frac{\gamma}{4|\widetilde s|}$, implying that
$$
\big\|(B_n^1(p)-p\,1_{n-1}1_n^\top+\widetilde s\,1_{n-1}1_n^\top) (x-y)\big\|_2
\leq \|B_n^1(p)-p\,1_{n-1}1_n^\top\|\,\frac{\gamma}{8 L}+|\widetilde s|\,\sqrt{n-1}\frac{\gamma}{4|\widetilde s|}
< \frac{\gamma\sqrt{n}}{2}
$$
everywhere on $\Event$. Hence,
\begin{align*}
\Prob&\big(\big\{\big\|(B_n^1(p)-p\,1_{n-1}1_n^\top+\widetilde s\,1_{n-1}1_n^\top) x\big\|_2\leq
\mbox{$\frac{\gamma}{2}$}\sqrt{n}
\mbox{ for some }x\in
S_\ell\big\}\cap\Event\big)\\
&\hspace{3cm}\leq
|\Net_\ell|\,\max\limits_{x\in\Net_\ell}\Prob\big\{
\big\|(B_n^1(p)-p\,1_{n-1}1_n^\top+\widetilde s\,1_{n-1}1_n^\top) x\big\|_2\leq \gamma\sqrt{n}\big\}\\
&\hspace{3cm}\leq {n\choose{\lfloor \delta n\rfloor}}\bigg(\frac{C'L}
{\gamma}\bigg)^{\lfloor \delta n\rfloor}
\bigg(\frac{e}{\widetilde \varepsilon}\bigg)^{\widetilde \varepsilon (n-1)}
(1-p)^{(n-1)(1-\widetilde \varepsilon)}.
\end{align*}

Observe further that for all vectors $x\in S^{n-1}$ with $\big|\sum_{i=1}^n x_i\big|\geq \frac{2L+2\gamma}{|\widetilde s|}$,
everywhere on the event $\Event$ we have
$$\big\|(B_n^1(p)-p\,1_{n-1}1_n^\top+\widetilde s\,1_{n-1}1_n^\top)
x\big\|_2\geq |\widetilde s|\,\sqrt{n-1}\,\Big|\sum_{i=1}^n x_i\Big|-L\sqrt{n}> \gamma\sqrt{n}.$$
Thus, everywhere on $\Event$, $\big\|(B_n^1(p)-p\,1_{n-1}1_n^\top+\widetilde s\,1_{n-1}1_n^\top)
x\big\|_2\geq \gamma\sqrt{n}$ for all $x\in S_\ell$ with $\ell\geq \frac{8(L+\gamma)}{\gamma}$
or $\ell\leq -\frac{8(L+\gamma)}{\gamma}-1$.
Combining all the above estimates, we obtain for some universal constant $C>0$:
\begin{align*}
\Prob\big\{&\big\|(B_n^1(p)-p\,1_{n-1}1_n^\top+\widetilde s\,1_{n-1}1_n^\top) x\big\|_2\leq
\mbox{$\frac{\gamma}{2}$}\sqrt{n}
\mbox{ for some }x\in
\comp_n(\delta,\nu)\big\}\\
&\leq \frac{C(L+\gamma)}{\gamma}{n\choose{\lfloor \delta n\rfloor}}\bigg(\frac{C'L}
{\gamma}\bigg)^{\lfloor \delta n\rfloor}
\bigg(\frac{e}{\widetilde \varepsilon}\bigg)^{\widetilde \varepsilon (n-1)}
(1-p)^{(n-1)(1-\widetilde \varepsilon)}+2^{-n}.
\end{align*}

It remains to note that by choosing $\widetilde\varepsilon=\widetilde\varepsilon(\varepsilon)$
sufficiently small, we can guarantee that the right hand side of the above inequality is less than
$$\frac{C(L+\gamma)}{\gamma}\bigg(1-p+\frac{\varepsilon}{2}\bigg)^{n-1}
+2^{-n}$$
for every $n\geq 2$.
Then the desired estimate will follow for all sufficiently large $n$ satisfying
$\frac{C(L+\gamma)}{\gamma}\big(1-p+\frac{\varepsilon}{2}\big)^{n-1}
+2^{-n}
\leq \big(1-p+\varepsilon\big)^{n}$.
\end{proof}

\section{Random averaging in $\ell_1(\Z)$}\label{s: averaging}

The main goal of this section is to provide upper bounds on the cardinalities of discretizations
of sets of vectors with a given threshold $\thres_p(\cdot,L)$,
discussed in the second part of Section~\ref{s: strategy}.
According to our ``inversion of randomness'', we consider a random vector uniformly distributed on a subset of the integer lattice $\Z^n$,
and want to show that with probability $1-e^{-\omega(n)}$ the scalar product of this vector with a vector
of independent Bernoulli($p$) variables has a small threshold value (with respect to the randomness of the Bernoulli vector).
First, we define the range of the random vector on the lattice.

Let $N,n\geq 1$ be some integers and let $\delta\in(0,1]$ and $K\geq 1$ be some real numbers.
We say that a subset $\mathcal A\subset\Z^n$ is {\it $(N,n,K,\delta)$--admissible} if
\begin{itemize}
\item $\mathcal A=A_1\times A_2\times\dots\times A_n$, where every $A_i$ ($i=1,2,\dots,n$)
is an origin-symmetric subset of $\Z$;
\item $A_i$ is an integer interval of cardinality at least $2N+1$ for every $i>\delta n$;
\item $A_i$ is a union of two integer intervals of total cardinality at least $2N$ and
$A_i\cap [-N,N]=\emptyset$ for all $i\leq \delta n$;
\item $|A_1|\cdot|A_2|\cdot\dots\cdot|A_n|\leq (KN)^n$;
\item $\max A_i < n\,N$ for all $1\leq i\leq n$.
\end{itemize}

\begin{Remark}
The condition $A_i\cap [-N,N]=\emptyset$ for $i\leq \delta n$, subject to appropriate rescaling,
is equivalent to the fact that the ``potential'' normal vectors we consider are $(\delta,\nu)$--incompressible,
hence at least $\lfloor \delta n\rfloor$ components of those vectors are separated from zero by $\nu/\sqrt{n}$.
\end{Remark}

Let $\mathcal A=A_1\times A_2\times\dots\times A_n\subset\Z^n$ be an $(N,n,K,\delta)$--admissible set,
and let $f(t)$ be any real valued function on $\Z$.
Fix any $p\in(0,1)$, and
assume that $X_1,X_2,\dots,X_n$ are independent integer random variables, where each $X_i$ is uniform in $A_i$.
For every $\ell\leq n$, we define a {\it random} function $f_{\mathcal A,p,\ell}$ by
%\begin{equation}\label{eq: averaging}
%f_{\mathcal A,p,i}(t):=(1-p)\, f_{\mathcal A,p,i-1}(t)+p\, f_{\mathcal A,p,i-1}(t+X_i),\quad
%t\in\Z,\quad i=1,2,\dots,n,
%\end{equation}
%where $f_{\mathcal A,p,0}:=f$.
%It is not difficult to see that, with such definition, we have
\begin{equation}\label{eq: fApell def}
f_{\mathcal A,p,\ell}(t):=
\Exp_b\,f\Big(t+\sum_{j=1}^\ell b_jX_j\Big)=
\sum\limits_{(v_j)_{j=1}^\ell\in\{0,1\}^\ell}
p^{\sum_j v_j} (1-p)^{\ell-\sum_j v_j}f\big(t+v_1X_1+\dots+v_\ell X_\ell\big),
\end{equation}
$t\in\Z$,
where $\Exp_b$ denotes the expectation with respect to the randomness of the vector $b=(b_1,\dots,b_n)$
with independent Bernoulli($p$) components.
The central statement of the section is the following theorem.
\begin{theorem}\label{th: averaging}
For any $\delta\in(0,1]$, $p\in(0,1/2]$, $\varepsilon\in(0,p)$, $K,M\geq 1$ there are
$n_{\text{\tiny\ref{th: averaging}}}=n_{\text{\tiny\ref{th: averaging}}}(\delta,\varepsilon,p,K,M)\geq 1$,
$\eta_{\text{\tiny\ref{th: averaging}}}=\eta_{\text{\tiny\ref{th: averaging}}}(\delta,\varepsilon,p,K,M)\in (0,1]$ 
depending on $\delta,\varepsilon,p,K,M$ and $L_{\text{\tiny\ref{th: averaging}}}=
L_{\text{\tiny\ref{th: averaging}}}(\delta,\varepsilon,p,K)>0$ depending {\bf only} on $\delta,\varepsilon,p,K$
(and not on $M$)
with the following property.
Take $n\geq n_{\text{\tiny\ref{th: averaging}}}$, $1\leq N\leq (1-p+\varepsilon)^{-n}$,
let $\mathcal A$ be an $(N,n,K,\delta)$--admissible set
and $f(t)$ be a non-negative function in $\ell_1(\Z)$ with $\|f\|_1=1$ and such that $\log_2 f$
is $\eta_{\text{\tiny\ref{th: averaging}}}$--Lipschitz.
Then, with $f_{\mathcal A,p,n}$ defined above, we have
$$\Prob\big\{\|f_{\mathcal A,p,n}\|_\infty> L_{\text{\tiny\ref{th: averaging}}}
(N\sqrt{n})^{-1}\big\}
\leq \exp(-M\, n).$$
\end{theorem}

The crucial feature of the theorem and the most important technical element of this paper,
is that the bound $L_{\text{\tiny\ref{th: averaging}}}
(N\sqrt{n})^{-1}$ on the $\ell_\infty$--norm of the averaged function does not depend on the parameter $M$
which controls the probability estimate.
In other words, for a given choice of $\delta,\varepsilon,p,K$, which determine the value of
$L_{\text{\tiny\ref{th: averaging}}}$, the probability bound can be made superexponentially small in $n$.

It is not difficult to check that with the only assumption $\|f\|_1=1$ on the function $f$ the above statement is false.\label{example ind}
For example, take $f$ to be the indicator of $\{0\}$, 
assume that $\mathcal A=\{-2N,-2N+1,\dots,-N-1,N+1,\dots,2N\}^{\lfloor \delta n\rfloor}\times \{-N,-N+1,\dots,N\}^{n-\lfloor \delta n\rfloor}$.
It can be shown that for any natural $q<N$, on the one hand, the event $\Event_q:=\{X_i\in q\,\Z,\;i=1,2,\dots,n\}$
has probability at least $(2q)^{-n}$, and, on the other hand, everywhere on $\Event_q$ we have
$\|f_{\mathcal A,p,n}\|_\infty\geq c_p q\,(N\sqrt{n})^{-1}$, because $f_{\mathcal A,p,n}$ is supported on $q\,\Z$
and (by standard concentration results) has most of its mass located within a (random) integer interval of length $O_p(N\sqrt{n})$.
Thus, the probability cannot be made superexponentially small in $n$ without taking $q$, hence the lower bound
for $\|f_{\mathcal A,p,n}\|_\infty\cdot (N\sqrt{n})$, to infinity.
The condition that the logarithm of the function is $\eta_{\text{\tiny\ref{th: averaging}}}$--Lipschitz,
employed in the theorem, is designed to rule out such situations.

Before proving the theorem, we shall consider the corollary which was (in a somewhat different form)
stated in the introduction as Theorem~B
and which will be used in our net-argument in the next section:

\begin{cor}\label{cor: anticoncentration}
Let $\delta,\varepsilon\in(0,1]$, $p\in(0,1/2]$, $K,M\geq 1$.
There exist $n_{\text{\tiny\ref{cor: anticoncentration}}}=n_{\text{\tiny\ref{cor: anticoncentration}}}(\delta,\varepsilon,p,K,M)\geq 1$
depending on $\delta,\varepsilon,p,K,M$ and $L_{\text{\tiny\ref{cor: anticoncentration}}}=
L_{\text{\tiny\ref{cor: anticoncentration}}}(\delta,\varepsilon,p,K)>0$ depending only on $\delta,\varepsilon,p,K$
(and not on $M$)
with the following property. Take $n\geq n_{\text{\tiny\ref{cor: anticoncentration}}}$, $1\leq N\leq (1-p+\varepsilon)^{-n}$,
and let $\mathcal A$ be an $(N,n,K,\delta)$--admissible set.
Further, assume that $b_1,b_2,\dots,b_n$ are i.i.d Bernoulli($p$) random variables.
Then
$$\bigg|\bigg\{x\in\mathcal A:\;\cf\Big(\sum_{i=1}^n b_i x_i,\sqrt{n}\Big)\geq L_{\text{\tiny\ref{cor: anticoncentration}}} N^{-1}
\bigg\}\bigg|
\leq e^{-M\,n}\,|\mathcal A|.$$
\end{cor}
\begin{proof}
Take $n\geq \max\big(n_{\text{\tiny\ref{th: averaging}}},1/\eta_{\text{\tiny\ref{th: averaging}}}^2\big)$,
and let $1\leq N\leq (1-p+\varepsilon)^{-n}$,
and $\mathcal A$ be an $(N,n,K,\delta)$--admissible set.
Define the function $f\in\ell_1(\Z)$ as
$$f(t):=\frac{1}{m_0}2^{-|t|/\sqrt{n}},\quad t\in\Z,$$
where $m_0=\sum_{t\in\Z}2^{-|t|/\sqrt{n}}$.
Obviously, $\|f\|_1=1$, and $\log_2 f$ is $n^{-1/2}$--Lipschitz, hence, by the assumptions on $n$,
$\log_2 f$ is $\eta_{\text{\tiny\ref{th: averaging}}}$--Lipschitz.

Applying Theorem~\ref{th: averaging} to $f$, we get
$$\Prob\big\{\|f_{\mathcal A,p,n}\|_\infty> L_{\text{\tiny\ref{th: averaging}}}
(N\sqrt{n})^{-1}\big\}
\leq \exp(-M\, n).$$
The definition of $f_{\mathcal A,p,n}$ allows to rewrite the above inequality as
\begin{align*}
\bigg|\bigg\{x\in\mathcal A:\; \sup\limits_{t\in\Z} 
\Exp_b\,f\Big(t+\sum_{j=1}^n b_jx_j\Big)
> L_{\text{\tiny\ref{th: averaging}}}
(N\sqrt{n})^{-1}\bigg\}\bigg|
\leq e^{-M\,n}\,|\mathcal A|.
\end{align*}
%where $b_1,b_2,\dots,b_n$ are independent Bernoulli($p$) variables.
On the other hand, since
$$f(t)\geq \frac{c}{\sqrt{n}}{\bf 1}_{[-\sqrt{n}-1,\sqrt{n}+1]}(t),\quad t\in\Z,$$
for some universal constant $c>0$, the last relation implies
\begin{align*}
\bigg|\bigg\{x\in\mathcal A:\; \sup\limits_{t\in\Z} 
\Exp_b\,{\bf 1}_{[-\sqrt{n}-1,\sqrt{n}+1]}\Big(t+\sum_{j=1}^n b_jx_j\Big)
> \frac{L_{\text{\tiny\ref{th: averaging}}}}{cN}\bigg\}\bigg|
\leq e^{-M\,n}\,|\mathcal A|.
\end{align*}
For every $t$ and $x=(x_1,x_2,\dots,x_n)$, the expression
$$
\Exp_b\,{\bf 1}_{[-\sqrt{n}-1,\sqrt{n}+1]}\Big(t+\sum_{j=1}^n b_jx_j\Big)
$$
is the probability that the random sum $t+\sum_{j=1}^n b_j x_j$ falls into the interval $[-\sqrt{n}-1,\sqrt{n}+1]$.
Thus, together with elementary relation $\sup\limits_{t\in\Z}\Prob\{|t+Y|\leq H+1\}\geq \cf(Y,H)$, valid for any $H\geq 0$
and any random variable $Y$, the previous inequality gives
\begin{align*}
\bigg|\bigg\{x\in\mathcal A:\; \cf\big(b_1x_1+\dots+b_nx_n,\sqrt{n})
> \frac{L_{\text{\tiny\ref{th: averaging}}}}{cN}\bigg\}\bigg|
\leq e^{-M\,n}\,|\mathcal A|.
\end{align*}
The statement follows.
\end{proof}

\bigskip

%Theorem~\ref{th: averaging} is obtained as a combination of several statements below.
In our proof of Theorem~\ref{th: averaging}, we will gradually improve delocalization estimates for the
functions $f_{\mathcal A,p,\ell}$.
Our first (simple) step --- Lemma~\ref{l: aux simple ac} --- is to obtain estimates
on the $\ell_1$--norm of the truncated function $f_{\mathcal A,p,\ell}\,{\bf 1}_I$ (with $\ell$ of order $n$)
for an arbitrary integer interval $I$ of length at most $N$. Upper bounds of the order $O_{p,\delta}(\|f\|_1\,/\sqrt{n})$
will follow from the L\'evy--Kolmogorov--Rogozin inequality stated in the preliminaries as Lemma~\ref{l: lkr}.
At the second step, Proposition~\ref{p: rough decay} below,
we prove a weaker version of Theorem~\ref{th: averaging} where the parameter $L$ is allowed to depend on $M$.
At the third step, we remove the dependence of $L$ on $M$
by using the Lipschitzness of $f$. A discussion of that part of the proof is given after Proposition~\ref{p: rough decay}.

\begin{lemma}\label{l: aux simple ac}
There is a universal constant $C_{\text{\tiny\ref{l: aux simple ac}}}>0$ with the following property.
Let $p\in(0,1)$, $\delta_0\in(0,1)$, let $f\in\ell_1(\Z)$ be a non-negative function with $\|f\|_1=1$, and let $\mathcal A$
be an $(N,n,K,\delta)$--admissible set for some parameters $N$, $\delta\in[\delta_0,1)$, $n\geq 1/\delta_0$ and $K$.
Further, let $\ell > \delta_0 n$. Then deterministically
$\sum\limits_{t\in I}f_{\mathcal A,p,\ell}(t)\leq \frac{C_{\text{\tiny\ref{l: aux simple ac}}}}{\sqrt{\delta_0 n\,\min(p,1-p)}}$
for any integer interval $I\subset\Z$ with $|I|\leq N$.
In turn, this implies
$$\sum\limits_{t\in J}f_{\mathcal A,p,\ell}(t)\leq \frac{2C_{\text{\tiny\ref{l: aux simple ac}}}|J|}{\sqrt{\delta_0 n\,\min(p,1-p)}N}$$
for any integer interval $J$ of cardinality {\it at least} $N$.
\end{lemma}
\begin{proof}
Let $X_1,\dots,X_\ell$ be the random variables from \eqref{eq: fApell def}.
Fix any realization of $X_1,\dots,X_\ell$ (so that $|X_i|> N$ for all $i\leq \delta_0 n$, by the definition of
an admissible set and since $\delta\geq \delta_0$),
and any integer interval $I$ of cardinality at most $N$.
Since
$$f_{\mathcal A,p,\ell}(t)=\sum\limits_{(v_i)_{i=1}^\ell\in\{0,1\}^\ell}
p^{\sum_i v_i}(1-p)^{\ell-\sum_i v_i}
f\big(t+v_1X_1+\dots+v_\ell X_\ell\big),
$$
we obtain
\begin{align*}
\sum\limits_{t\in I}f_{\mathcal A,p,\ell}(t)&=
\sum\limits_{(v_i)_{i=1}^\ell\in\{0,1\}^\ell}\;
\sum\limits_{t\in I}p^{\sum_i v_i}(1-p)^{\ell-\sum_i v_i} f\big(t+v_1X_1+\dots+v_\ell X_\ell\big)\\
&=\sum\limits_{(v_i)_{i=1}^\ell\in\{0,1\}^\ell}\;
\sum\limits_{t\in \Z}p^{\sum_i v_i}(1-p)^{\ell-\sum_i v_i} f(t){\bf 1}_{I+v_1X_1+\dots+v_\ell X_\ell}(t)\\
&=\Big\|f\,\sum\limits_{(v_i)_{i=1}^\ell\in\{0,1\}^\ell}p^{\sum_i v_i}(1-p)^{\ell-\sum_i v_i} {\bf 1}_{I+v_1X_1+\dots+v_\ell X_\ell}\Big\|_{1}.
\end{align*}
For any $t\in\Z$,
$$\sum\limits_{(v_i)_{i=1}^\ell\in\{0,1\}^\ell}
p^{\sum_i v_i}(1-p)^{\ell-\sum_i v_i}
{\bf 1}_{I+v_1X_1+\dots+v_\ell X_\ell}(t)
=\Prob\big\{b_1X_1+\dots+b_\ell X_\ell\in t-I \vert X_1,\dots,X_\ell\big\}.$$
where $b_1,\dots,b_\ell$ are Bernoulli($p$) random variables jointly independent with $X_1,\dots,X_\ell$.
It remains to note that the L\'evy--Kolmogorov--Rogozin inequality (Lemma~\ref{l: lkr}), together with the condition
$|X_i|> N$ for all $i\leq \delta_0 n$, implies that
for every $t\in\Z$,
$$
\Prob\big\{b_1X_1+\dots+b_\ell X_\ell\in t-I \;\vert\;X_1,\dots,X_\ell\big\}
\leq \frac{C}{\sqrt{\delta_0 n\,\min(p,1-p)}},
$$
for some universal constant $C>0$.
The result follows.
\end{proof}

\begin{prop}\label{p: rough decay}
For any $M>0$, $p\in(0,1/2]$, $\delta\in(0,1)$ and $\varepsilon\in(0,p)$ there are $L_{\text{\tiny\ref{p: rough decay}}}
=L_{\text{\tiny\ref{p: rough decay}}}(M,p,\delta,\varepsilon)>0$
and $n_{\text{\tiny\ref{p: rough decay}}}=n_{\text{\tiny\ref{p: rough decay}}}(M,p,\delta,\varepsilon)\in\N$
(depending on $M$, $p$, $\delta$
and $\varepsilon$) with the following property.
Let $f\in\ell_1(\Z)$ be a non-negative function with $\|f\|_1=1$, let $n\geq n_{\text{\tiny\ref{p: rough decay}}}$,
$n/2\leq \ell\leq n$,
and let
$\mathcal A$ be an $(N,n,K,\delta)$--admissible set for some parameters $N\leq 2^n$
and $K>0$.
Then
$$\Prob\big\{\|f_{\mathcal A,p,\ell}\|_\infty> 
\max\big(L_{\text{\tiny\ref{p: rough decay}}}(N\sqrt{n})^{-1},  (1-p+\varepsilon)^{\ell}\,\|f\|_\infty\big)\big\}
\leq \exp(-M n),$$
where $f_{\mathcal A,p,\ell}$ is defined by \eqref{eq: fApell def}.
\end{prop}

The crucial difference between the above statement and Theorem~\ref{th: averaging}
is that $L_{\text{\tiny\ref{p: rough decay}}}$ in the proposition
is allowed to depend on $M$.
The proof essentially follows by estimating probabilities that $
f_{\mathcal A,p,\ell}(t)> 
\max\big(L_{\text{\tiny\ref{p: rough decay}}}(N\sqrt{n})^{-1},  (1-p+\varepsilon)^{\ell}\,\|f\|_\infty\big)$
for a fixed $t\in\Z$ and taking the union bound over $t$,
although the actual argument is more involved.
We will need the following definitions.

Let $R>0$ be a parameter, let
$N$, $\mathcal A$, $f$, $\ell$ and $p$ be as in the above proposition, and let $m\in\{1,2,\dots,\ell\}$.
We say that a point $t\in\Z$ {\it decays at time $m$} if
$$f_{\mathcal A,p,m-1}(t+X_m)\leq\frac{R}{N\sqrt{n}}
\quad\mbox{ and }\quad f_{\mathcal A,p,m-1}(t-X_m)\leq\frac{R}{N\sqrt{n}}.$$
Further, given any $t\in\Z$ and a sequence
$(v_i)_{i=1}^\ell\in\{0,1\}^\ell$, {\it the descendant sequence} for $t$
with respect to $(v_i)_{i=1}^\ell$ is a random sequence $(t_i)_{i=0}^\ell$,
where $t_i=t-\sum_{j=1}^i v_j X_j$, $1\leq i\leq \ell$ (and where we set $t_0:=t$).
The connection of the above statement with these definitions
is provided by the following fact: the event that the $\ell_\infty$--norm of $f_{\mathcal A,p,\ell}$
is ``large'' is contained within the event that there exists a descendant sequence such that a proportional
number of its elements do not decay.
More precisely, we have
\begin{lemma}\label{l: aux 09019847294}
Let $N$, $\mathcal A$, $f$, $\ell$, $\varepsilon$ and $p$ be as in Proposition~\ref{p: rough decay},
let $L>0$, and set $R:=\frac{\varepsilon L}{2p}$.
Define event $\Event$ as the subset the probability space such that
there exists a sequence
$(v_i)_{i=1}^\ell\in\{0,1\}^\ell$ and a point
$t\in\Z$ so that the descendant sequence $(t_i)_{i=0}^\ell$ for $t$ with respect to $(v_i)_{i=1}^\ell$
satisfies
\begin{equation}\label{eq: aux 4733308t3}
\big|\big\{1\leq i\leq \ell:\;t_{i-1}\mbox{ does not decay at time $i$}\big\}\big|\geq 
-\frac{n\,\log\big((1-p+\varepsilon)/(1-p+\varepsilon/2)\big)}{2\log\big(1-p+\varepsilon/2\big)}.
\end{equation}
Then $\Event\supset \big\{\|f_{\mathcal A,p,\ell}\|_\infty>
\max\big(L(N\sqrt{n})^{-1}, (1-p+\varepsilon)^{\ell}\,\|f\|_\infty\big)\big\}$.
\end{lemma}
\begin{proof}
Fix a realization of $X_1,\dots,X_\ell$ such that
$$\|f_{\mathcal A,p,\ell}\|_\infty>
\max\big(L(N\sqrt{n})^{-1}, (1-p+\varepsilon)^{\ell}\,\|f\|_\infty\big)$$
(if such a realization does not exist then there is nothing to prove).
We will construct a sequence of integers $(t_i)_{i=0}^\ell$ inductively
in inverse order as follows. Take $t_\ell$ to be any integer such that
$f_{\mathcal A,p,\ell}(t_\ell)>
\max\big(L(N\sqrt{n})^{-1}, (1-p+\varepsilon)^{\ell}\,\|f\|_\infty\big)$.
At $(\ell-i+1)$--st step ($1\leq i\leq \ell$) we assume that $t_i$
has been defined, and satisfies
$f_{\mathcal A,p,\ell}(t_i)>
\max\big(L(N\sqrt{n})^{-1}, (1-p+\varepsilon)^{\ell}\,\|f\|_\infty\big)$. 
In view of the relation
\begin{equation}\label{eq: aux averaging}
f_{\mathcal A,p,i}(t):=(1-p)\, f_{\mathcal A,p,i-1}(t)+p\, f_{\mathcal A,p,i-1}(t+X_i),\quad
t\in\Z,
\end{equation}
which follows immediately from the definition of $f_{\mathcal A,p,i}$, we get that
$f_{\mathcal A,p,i-1}(t_i+v_iX_i)\geq 
f_{\mathcal A,p,i}(t_i)$ for some $v_i\in\{0,1\}$. Then we set $t_{i-1}:=t_i+v_i X_i$.

Clearly, the sequence $(t_i)_{i=0}^\ell$ constructed this way, is the descendant sequence
for $t_0$ with respect to $(v_i)_{i=1}^\ell$, which satisfies the conditions
\begin{itemize}
\item[(a)] $f_{\mathcal A,p,i-1}(t_{i-1})\geq f_{\mathcal A,p,i}(t_i)$ for all $1\leq i\leq \ell$;
\item[(b)] $f_{\mathcal A,p,\ell}(t_\ell)>
\max\big(L(N\sqrt{n})^{-1}, (1-p+\varepsilon)^{\ell}\,\|f\|_\infty\big)$.
\end{itemize}
We will show that these conditions imply \eqref{eq: aux 4733308t3}.
Assume that $1\leq i\leq\ell$ is such that $t_{i-1}$
decays at time $i$.
According to \eqref{eq: aux averaging} and the relation between $t_i$ and $t_{i-1}$, we have
\begin{align*}
f_{\mathcal A,p,i}(t_i)&=(1-p) f_{\mathcal A,p,i-1}(t_i)+p\,f_{\mathcal A,p,i-1}(t_i+X_i)\\
&=(1-p)f_{\mathcal A,p,i-1}(t_{i-1}-v_iX_i)+p\,f_{\mathcal A,p,i-1}(t_{i-1}+(1-v_i)X_i).
\end{align*}
By our definition of decay at time $i$,
both $f_{\mathcal A,p,i-1}(t_{i-1}+X_i)$ and $f_{\mathcal A,p,i-1}(t_{i-1}-X_i)$
are less than $\frac{R}{N\sqrt{n}}$, hence less than $\frac{\varepsilon}{2p}\,f_{\mathcal A,p,i-1}(t_{i-1})$,
by the relation between $L$ and $R$ and conditions (a), (b).
Thus, one of the values $f_{\mathcal A,p,i-1}(t_{i-1}-v_iX_i)$ or
$f_{\mathcal A,p,i-1}(t_{i-1}+(1-v_i)X_i\big)$ is at most $\frac{\varepsilon}{2p}\,f_{\mathcal A,p,i-1}(t_{i-1})$
while the other is equal to $f_{\mathcal A,p,i-1}(t_{i-1})$.
This gives
$$f_{\mathcal A,p,i}(t_i)\leq \Big(\frac{\varepsilon}{2p}\cdot p+1-p\Big)\,f_{\mathcal A,p,i-1}(t_{i-1}).$$
Applying the last relation for all $i$ where there is a decay and using the monotonicity of
the sequence $\big(f_{\mathcal A,p,j}(t_j)\big)_{j=0}^\ell$, we get
for $u=|\{1\leq i\leq \ell:\;t_{i-1}\mbox{ decays at time $i$}\}|$:
$$(1-p+\varepsilon)^{\ell}\,\|f\|_\infty< f_{\mathcal A,p,\ell}(t_\ell)\leq (1-p+\varepsilon/2)^u\,\|f\|_\infty,$$
whence
$$(1-p+\varepsilon/2)^{\ell-u}<\big((1-p+\varepsilon/2)/(1-p+\varepsilon)\big)^{n/2}.$$
This implies the required lower bound for $\ell-u=|\{1\leq i\leq \ell:\;t_{i-1}\mbox{ does not decay at time $i$}\}|$.
\end{proof}

\begin{proof}[{Proof of Proposition~\ref{p: rough decay}}]
Let $L>0$ be a parameter to be chosen later.
Set
$$\eta:=\min\bigg(\delta
,-\frac{\log\big((1-p+\varepsilon)/(1-p+\varepsilon/2)\big)}{2\log\big(1-p+\varepsilon/2\big)}\bigg);
\quad R:=\frac{\varepsilon L}{2p}.$$
We will assume that $\eta n/2\geq 1$.
Let $X_1,X_2,\dots,X_\ell$ be independent random variables,
each $X_i$ uniform on $A_i$, where $\mathcal A=A_1\times A_2\times\dots\times A_n$.

The proposition follows by applying Lemma~\ref{l: aux 09019847294} and a union bound.
Observe that
for any point $t\in\Z$ such that the last element of
a descendant sequence $(t_i)_{i=0}^\ell$ (with respect to some sequence in $\{0,1\}^\ell$ and with $t_0=t$)
satisfies
$f_{\mathcal A,p,\ell}(t_\ell)>(N\sqrt{n})^{-1}$,
we have
$$t\in \big\{s\in\Z:\;f(s)> (N\sqrt{n})^{-1}\big\}
+(A_1\cup\{0\}+A_1\cup\{0\})+\dots+(A_\ell\cup\{0\}+A_\ell\cup\{0\}).$$
Indeed, the definition of the descendant sequence implies that for some $(\widetilde v_i)_{i=1}^\ell\in\{0,1\}^\ell$,
$$t=t_\ell+\widetilde v_1 X_1+\dots+\widetilde v_\ell X_\ell\in
t_\ell +A_1\cup\{0\}+\dots+ A_\ell\cup\{0\},$$
while at the same time
the condition $f_{\mathcal A,p,\ell}(t_\ell)>(N\sqrt{n})^{-1}$ and the definition of $f_{\mathcal A,p,\ell}$
implies that $f(t_\ell+x_1+x_2\dots+x_\ell)>(N\sqrt{n})^{-1}$ for some $x_i\in A_i\cup\{0\}$, $i=1,\dots,\ell$,
hence
\begin{align*}
&t_\ell\in \big\{s\in\Z:\;f(s)> (N\sqrt{n})^{-1}\big\}-A_1\cup\{0\}-\dots-A_\ell\cup\{0\}\\
&\hspace{3cm}=
\big\{s\in\Z:\;f(s)> (N\sqrt{n})^{-1}\big\}+A_1\cup\{0\}+\dots+A_\ell\cup\{0\}.
\end{align*}

Set
$$D:=\big\{s\in\Z:\;f(s)> (N\sqrt{n})^{-1}\big\}
+(A_1\cup\{0\}+A_1\cup\{0\})+\dots+(A_\ell\cup\{0\}+A_\ell\cup\{0\}),$$
and observe that, in view of the upper bound on $\max A_i$'s from the definition of an admissible set,
and the assumption $\|f\|_1=1$,
\begin{align*}
|D|\leq N\sqrt{n}\,\big|(A_1\cup\{0\}+A_1\cup\{0\})+\dots+(A_\ell\cup\{0\}+A_\ell\cup\{0\})
\big|\leq 4N\sqrt{n}\, \ell\,nN\leq 4N^2 n^{5/2}.
\end{align*}
Set $H:=\eta n$.
Then, with the event $\Event$ defined in Lemma~\ref{l: aux 09019847294}, we can write
\begin{align*}
\Prob(\Event)
&\leq 2^\ell|D|\sup\limits_{t\in D,\;(v_i)_{i=1}^\ell\in\{0,1\}^\ell}\Prob\big\{
\mbox{The descendant sequence $(t_i)_{i=0}^\ell$ for $t$ w.r.t $(v_i)_{i=1}^\ell$}\\
&\hspace{4.8cm}\mbox{satisfies $|\{1\leq i\leq \ell:\;t_{i-1}\mbox{ does not decay at $i$}\}|\geq H$}
\big\}\\
&\leq 2^{\ell+2}N^2 n^{5/2}{n\choose {\lceil H\rceil}}
\sup\limits_{\substack{I\subset[\ell],\;|I|=\lceil H\rceil\\ t\in D,
\;(v_i)_{i=1}^\ell\in\{0,1\}^\ell}}\Prob\big\{
\mbox{For descendant sequence $(t_i)_{i=0}^\ell$ w.r.t $(v_i)$,}\\
&\hspace{7cm}\mbox{$t_{i-1}$ does not decay for all $i\in I$}
\big\}.
\end{align*}
Finally, fix any $I\subset[\ell]$ with $|I|=\lceil H\rceil$,
$t\in D$ and $(v_i)_{i=1}^\ell\in\{0,1\}^\ell$.
Let $(t_i)_{i=0}^\ell$ be the (random) descendant sequence for $t$ with respect to $(v_i)$
(note that $t_i$ is measurable w.r.t.\ $X_1,\dots,X_i$).
Take any $i\in I$ with $i-1> H/2$. 
Conditioned on any realization of $X_1,\dots,X_{i-1}$,
the variable $t_{i-1}+X_i$ is uniform on $t_{i-1}+A_i$, and
\begin{align*}
\Exp\big(f_{\mathcal A,p,i-1}(t_{i-1}+X_i)\;\vert\;X_1,\dots,X_{i-1}\big)
&=\frac{1}{|A_i|}\sum\limits_{s\in t_{i-1}+A_i}f_{\mathcal A,p,i-1}(s)\\
&\leq \frac{4}{N}
\frac{C_{\text{\tiny\ref{l: aux simple ac}}}}{\sqrt{p\eta n/2}},
\end{align*}
where at the last step we applied Lemma~\ref{l: aux simple ac} with $\delta_0:=\eta/2$
and used that $A_i$ is either an integer interval or a union of two integer intervals.
The same estimate is valid for
$$\Exp\big(f_{\mathcal A,p,i-1}(t_{i-1}-X_i)\;\vert\;X_1,\dots,X_{i-1}\big).$$
Hence, by Markov's inequality,
\begin{align*}
\Prob&\big\{\mbox{$t_{i-1}$ does not decay at $i$}\;\vert\;X_1,\dots,X_{i-1}\big\}\\
&=
\Prob\Big\{
\mbox{$f_{\mathcal A,p,i-1}(t_{i-1}+X_i)>\frac{R}{N\sqrt{n}}$ or $f_{\mathcal A,p,i-1}(t_{i-1}-X_i)>\frac{R}{N\sqrt{n}}$}
\;\Big\vert\;X_1,\dots,X_{i-1}\Big\}\\
&\leq \frac{8}{N} 
\frac{C_{\text{\tiny\ref{l: aux simple ac}}}}{\sqrt{p\eta n/2}} \frac{N\sqrt{n}}{R}
= \frac{8C_{\text{\tiny\ref{l: aux simple ac}}}}{\sqrt{p\eta/2}R}.
\end{align*}
Applying this estimate for all $i\in I\setminus [1,H/2+1]$,
we obtain
\begin{align*}
\Prob\big\{
\mbox{For desc.\ sequence $(t_i)_{i=0}^\ell$, $t_{i-1}$ doesn't decay at $i$ for all $i\in I$}
\big\}
\leq \bigg(\frac{8C_{\text{\tiny\ref{l: aux simple ac}}}}{\sqrt{p\eta/2}R}\bigg)^{\lceil H\rceil-H/2-2},
\end{align*}
whence
$$
\Prob(\Event)
\leq 2^{\ell+2}N^2 n^{5/2}{n\choose {\lceil H\rceil}}
\bigg(\frac{16C_{\text{\tiny\ref{l: aux simple ac}}}p}{\sqrt{p\eta/2}\,\varepsilon L}\bigg)^{\lceil H\rceil-H/2-2},
$$
where, we recall, $H=\eta n$.
Finally, we observe that by choosing $L=L(M,p,\delta,\varepsilon)$ large enough, we can make the last expression less than $\exp(-M n)$
for all sufficiently large $n$.
This completes the proof of the proposition.
\end{proof}

The above result is too weak to be useful for our purposes.
The rest of the section is devoted to ``refining'' the proposition
by removing the dependence on $M$ from the lower bound on the $\ell_\infty$--norm
of the averaged function.
%Note that, while the above proposition works under weak assumptions on $f$,
%below, in our refinement procedure, we will need to additionally assume that
%$\log f$ is $\eta$--Lipschitz for a small enough parameter $\eta$.

Let us informally describe the idea behind the argument and provide some simple examples.
The magnitude of the $\ell_\infty$--norm of $f_{\mathcal A,p,n}$ essentially depends on how
efficient in removing spikes is the averaging step given by the relation
$f_{\mathcal A,p,i}(t)=(1-p)\, f_{\mathcal A,p,i-1}(t)+p\, f_{\mathcal A,p,i-1}(t+X_i)$. 
%which follows directly from the definition of $f_{\mathcal A,p,i}$'s.
One may hope that if at every step $i$, the number of spikes (coordinates with large magnitudes) is decreased significantly
with a probability close to one then the resulting function $f_{\mathcal A,p, n}$ would have a small 
$\ell_\infty$--norm with a very large probability (superexponentially close to one).

For a moment, it will be convenient to drop the assumption of a bounded $\ell_1$--norm.
Consider a family of functions $g_{N,d,I,\eta}$ on $\Z$, indexed by natural numbers $N,d$,
an integer interval $I$, and $\eta>0$, and defined as
\begin{equation*}\label{pge: desc 1}
g_{N,d,I,\eta}(t):=\exp\big(-\eta\,\dist(t,I+d\,\Z)\big),\quad t\in\Z,
\end{equation*}
where we impose the following restrictions on parameters:
\begin{itemize}
\item $N\geq d$;
\item The function $g_{N,d,I,\eta}$ is ``essentially non-constant'' in the sense that $\|g_{N,d,I,\eta}{\bf 1}_{J}\|_1
\leq \frac{1}{2}|J|$ for any integer interval $J$ of length at least $N$.
\end{itemize}
Note that $\log g_{N,d,I,\eta}$ is $\eta$--Lipschitz and that the second assumption implies $|I|\leq d/2$.
Assume that a random variable $X$ is uniformly distributed on $\{0,1,\dots,N\}$,
and define the random average
$$
g^{av}_{N,d,I,\eta}(t):=\frac{1}{2}g_{N,d,I,\eta}(t)+\frac{1}{2}g_{N,d,I,\eta}(t+X),\quad t\in\Z.
$$
We are interested in estimating the proportion $\mathcal R_{N,d,I,\eta}$ of spikes preserved by the averaging; with
$$\mathcal R_{N,d,I,\eta}:=\lim\limits_{k\to\infty}\frac{|\{t\in\Z\cap [-k,k]:\;g^{av}_{N,d,I,\eta}(t)=1\}|}
{|\{t\in\Z\cap [-k,k]:\;g_{N,d,I,\eta}(t)=1\}|}.$$
A simple computation taking into account the condition $|I|\leq d/2$, gives
$$
\Prob\big\{1-\mathcal R_{N,d,I,\eta}\leq \varepsilon\big\}
= \Theta\Big(\frac{\varepsilon|I|}{d}+\frac{1}{d}\Big),\quad \varepsilon\in(0,1/2]
$$
and, for $\varepsilon=0$,
$$
\Prob\big\{1-\mathcal R_{N,d,I,\eta}=0\big\}=\Theta\Big(\frac{1}{d}\Big).
$$
Thus, the {\it efficiency} of the averaging, i.e.\ the small ball probability estimate for $1-\mathcal R_{N,d,I,\eta}$,
is influenced by the magnitude of $d$ or, equivalently, the length $d-|I|$ of the ``valleys'' separating
the clusters of spikes in $g_{N,d,I,\eta}$.
Now, let us discuss how this is related to the Lipschitzness of the logarithm.
It is not difficult to check that, in order to satisfy the condition of being ``essentially non-constant'',
we must choose $d$ at least of order $1/\eta$. Thus, the smaller $\eta$ is,
the wider the valleys between the clusters of spikes, and the stronger the small ball probability
estimates for $1-\mathcal R_{N,d,I,\eta}$ must be.
In a sense, the Lipschitzness of the logarithm of $g_{N,d,I,\eta}$, together with the essential non-constantness, affects
the averaging indirectly, by influencing the structure of spikes and valleys.

In our actual model, a similar phenomenon holds, although the argument is more complicated,
first, because the pattern of spikes does not have to be as regular as in the above example, second,
because the spikes are defined as points where the function exceeds a certain threshold
rather than points where it takes a specific value.
Our measurement of the efficiency of the averaging is more complicated compared to the above example.
For a function with relatively many spikes, we compare the $\ell_2$--norms of the original function and the average.
A crucial step towards proving Theorem~\ref{th: averaging} is the following proposition.

\begin{prop}\label{prop: ell 2 update}
Let $R>0$, $p\in(0,1)$, $\mu\in(0,1/64]$ and $N\in\N$.
Further, assume that
$g_1,g_2$ are non-negative functions in $\ell_1(\Z)$, and $g_1$ satisfies the following conditions:
\begin{itemize}
\item $\log_2 g_1$ is $\mu^4$--Lipschitz;
\item $\sum\limits_{t\in I}g_1(t)\leq RN$ for any integer interval $I$ of cardinality $N$;
\item There is interval $I_0\subset\Z$ with $|I_0|=N$, such that $|\{t\in I_0:\;g_1(t)\geq 8R\}|\geq \mu N$.
\end{itemize}
Let $Y$ be a random variable uniformly distributed on an integer interval $J$ of cardinality at least $N$.
Then
$$\Prob\Big\{\mbox{$\big\|(1-p)\,g_1(\cdot)+p\,g_2(\cdot+Y))\big\|_2^2\leq \big((1-p)\|g_1\|_2^2+p\|g_2\|_2^2\big)
-c_{\text{\tiny\ref{prop: ell 2 update}}} p(1-p)\mu^6 R^2 N
$}\Big\}\geq 1-C_{\text{\tiny\ref{prop: ell 2 update}}}\mu.$$
Here, $C_{\text{\tiny\ref{prop: ell 2 update}}},c_{\text{\tiny\ref{prop: ell 2 update}}}>0$
are universal constants.
\end{prop}

Before proving the proposition, we consider two lemmas.

\begin{lemma}\label{l: aux -98745220985-32}
Let $f,g\in\ell_2(\Z)$, and assume that $\kappa>0$ and $k\in\N$ are such that
$$\big|\big\{t\in\Z:\;|f(t)-g(t)|\geq \kappa\big\}\big|\geq k.$$
Let $p\in(0,1)$.
Then $\big\|p f+(1-p)g\big\|_2^2\leq \big(p\|f\|_2^2+(1-p)\|g\|_2^2\big)-p(1-p)\kappa^2k$.
\end{lemma}
\begin{proof}
For any $t\in\Z$ we have
\begin{align*}
\big(&p f(t)+(1-p)g(t)\big)^2\\
&=
p f(t)^2+(1-p)g(t)^2-\big(p(1-p)f(t)^2-2p(1-p)f(t)g(t)+p(1-p)g(t)^2\big)\\
&=
p f(t)^2+(1-p)g(t)^2-p(1-p)\big(f(t)-g(t)\big)^2,
\end{align*}
which implies the estimate.
\end{proof}

\begin{lemma}\label{l: aux averaging}
Let $f,g\in\ell_1(Z)$, and $\delta,\kappa>0$.
Further, assume that $I\subset\Z$ is an integer interval
and $I_1\cup I_2\cup I_3=I$ is a partition of $I$ into three subsets (not necessarily subintervals) such that
$|I_3|\in \big[\delta|I|/2,\delta |I|\big]$, $|I_2|\leq \delta|I|$, and $f(t_1)\geq \kappa+f(t_3)$
for all $t_1\in I_1$ and $t_3\in I_3$. Further, assume that $X$ is an integer random variable uniformly distributed
on an interval $J\subset\Z$ of cardinality at least $|I|$. Then
$$\Prob\big\{\big|\big\{t\in I:\;|f(t)-g(t+X)|\geq\kappa/2\big\}\big|<\delta |I|/4\big\}\leq 64\delta.$$ 
\end{lemma}
\begin{proof}
Without loss of generality, $\delta\leq 1/64$.
Fix any subinterval $\widetilde J\subset J$ of cardinality at least $|I|/2$ and at most $|I|$.
We will prove the probability estimate under the condition that $X$ belongs
to $\widetilde J$.
Then the required result will easily follow by partitioning $J$ into subintervals and combining
estimates for corresponding conditional probabilities.

Set
$$w_3:=\max\limits_{t_3\in I_3}f(t_3);\quad w_1:=\min\limits_{t_1\in I_1}f(t_1),$$ 
and define
$$Q:=\big\{i\in \widetilde J:\;\big|\big\{t\in I:\;g(t+i)\leq (w_1+w_3)/2\big\}\big|\leq 4\delta|I|\big\}.$$
Observe that, in view of the assumption $w_1\geq w_3+\kappa$,
for any point $i\in \widetilde J\setminus Q$ we have
$$
\big|\big\{t\in I:\;|f(t)-g(t+i)|\geq\kappa/2\big\}\big|\geq 4\delta|I|-|I_2|-|I_3|\geq 2\delta|I|.
$$
Thus, if $Q=\emptyset$ then, conditioned on $X\in\widetilde J$,
$\big|\big\{t\in I:\;|f(t)-g(t+X)|\geq\kappa/2\big\}\big|<\delta |I|/4$ holds with probability zero,
and the statement follows. Below, we assume that $Q\neq\emptyset$.

Set $S:=\{\min Q,\min Q+1,\dots,\max Q\}$.
Since $|\widetilde J|\leq|I|$, we have $S+I= (\min Q+I)\cup (\max Q+I)$, whence
\begin{align*}
\big|&\big\{s\in S+I:g(s)\leq (w_1+w_3)/2\big\}\big|\\
&\leq
\big|\big\{t\in I:g(t+\min Q)\leq (w_1+w_3)/2\big\}\big|
+\big|\big\{t\in I:g(t+\max Q)\leq (w_1+w_3)/2\big\}\big|\\
&\leq 8\delta |I|.
\end{align*}
%We say that a point $t\in I_3$ is {\it matched at $i\in S$}
%if $g(t+i)\leq \frac{w_1+w_3}{2}$.
The above estimate immediately gives
$$\Big|\Big\{(t,i)\in I_3\times S:\;%\mbox{$t$ is matched at $i$}
g(t+i)\leq \frac{w_1+w_3}{2}\Big\}\Big|
\leq 8\delta |I|\cdot |I_3|\leq 8\delta^2 |I|^2.
$$
Hence, the number of points $i\in S$ such that
\begin{equation}\label{eq: aux 057605353}
\Big|\Big\{t\in I_3:\;%\mbox{$t$ is matched at $i$}
g(t+i)\leq \frac{w_1+w_3}{2}\Big\}\Big|> \delta |I|/4,
\end{equation}
is at most $32\delta |I|$.
On the other hand,
for every $i\in S$ such that \eqref{eq: aux 057605353} does not hold,
we clearly have
$$\big|\big\{t\in I:\;|f(t)-g(t+i)|\geq \kappa/2\big\}\big|\geq |I_3|-\delta |I|/4\geq \delta |I|/4.$$

Summarizing, we obtain
$$\Big|\Big\{i\in \widetilde J:\;
\big|\big\{t\in I:\;|f(t)-g(t+i)|\geq \kappa/2\big\}\big|<\delta |I|/4\Big\}\Big|\leq 32\delta |I|,
$$
whence
$$
\Prob\big\{\big|\big\{t\in I:\;|f(t)-g(t+X)|\geq\kappa/2\big\}\big|<\delta |I|/4\;\vert\;X\in\widetilde J\big\}\leq 64\delta.
$$
The result follows.
\end{proof}

\begin{proof}[{Proof of Proposition~\ref{prop: ell 2 update}}]
Let $\delta:=8\mu$, $\varepsilon:=\mu^4$ and $\widetilde I:=I_0+\{0,1,\dots,N\}$, so that $|\widetilde I|=2N$.
It is not difficult to see that there is a real interval of the form $(a,2^{\mu^2}a]$, where $4R\leq a\leq 2^{-\mu^2}\cdot 8R$
and such that
\begin{equation}\label{eq: aux 30872976598}
\Big|\Big\{t\in\widetilde I:\;g_1(t)\in \big(a,2^{\mu^2}a\big]\Big\}\Big|\leq \frac{2N}{\lfloor 1/\mu^2\rfloor}.
\end{equation}
We will inductively construct a finite sequence of integer intervals $I'_1,I'_2,\dots,I'_h$ as follows.

At the first step, let $t_1^\ell:=\min\{t\in\widetilde I:\;g_1(t)\geq 2^{\mu^2}a\}$,
$$t_1^r:=\max\big\{t\in\widetilde I:\;t\geq t_1^\ell;\;|\{s\in \{t_1^\ell,\dots,t\}:\;g_1(s)\leq a\}|\leq \delta (t-t_1^\ell+1)\big\},$$
and define $I'_1:=\{t_1^\ell, t_1^\ell+1,\dots, t_1^r\}$ (note that by the definition of $I_0$, $t_1^\ell$ exists).
In words, we choose $t_1^r$ to be the largest integer in $\widetilde I$ such that the number
of the elements $s\in I'_1$ corresponding to ``small'' values $g_1(s)\leq a$, is at most $\delta|I'_1|$.
If $\max I'_1\geq \max I_0$ or
if $g_1(t)< 2^{\mu^2}a$ for all $t^r_1=\max I'_1< t\leq \max I_0$ then we set $h:=1$ and complete the process.
Otherwise, we go to the second step.

At $k$-th step, $k>1$, we define $t_k^\ell> I'_{k-1}$ to be the smallest integer in $(\max I'_{k-1},\infty)$
such that $g_1(t_k^\ell)\geq 2^{\mu^2} a$
(the previous step of the construction guarantees that such $t_k^\ell$ exists and belongs to $I_0$).
We set $t_k^r:=\max\big\{t\in \widetilde I:\;t\geq t_k^\ell;\;|\{s\in \{t_k^\ell,\dots,t\}:\;g_1(s)\leq a\}|\leq \delta (t-t_k^\ell+1)\big\}$,
and $I'_k:=\{t_k^\ell,t_k^\ell+1,\dots,t_k^r\}$.
If $\max I'_k\geq \max I_0$ or
if $g_k(t)< 2^{\mu^2}a$ for all $t^r_k=\max I'_k< t\leq \max I_0$ then set $h:=k$ and complete,
otherwise go to the next step.

\medskip

Next, we observe some important properties of the constructed sequence.
\begin{itemize}
\item[(a)] The left-points of all intervals are contained in $I_0$, and the union $\bigcup_{k=1}^h I'_k$
contains the set $\{t\in I_0:\;g_1(t)\geq 2^{\mu^2} a\}$; in particular, cardinality of the union is at least $\mu N$.

\item[(b)] The cardinality of any interval $I'_k$ cannot exceed $N$ since our assumption on the function $g_1$,
together with the definition of $I'_k$, gives
$$2R|I'_k|\leq a|I'_k|/2\leq a (|I'_k|-\delta |I'_k|)  \leq \sum\limits_{t\in I'_k}g_1(t)\leq \sum\limits_{t\in \widetilde I}g_1(t)\leq 2RN.$$
In particular, this implies that $\max I'_h$ is {\it strictly less} than $\max \widetilde I$.

\item[(c)] The condition that $\log_2 g_1$ is $\varepsilon$--Lipschitz implies that
for any $k\leq h$, $|I'_k|\geq \lfloor \mu^2/\varepsilon\rfloor>\frac{1}{4\mu}$. Indeed, since $g_1(t+1)\geq 2^{-\varepsilon}g_1(t)$
for all $t\in\Z$, we have $g_1(t)>2^{-\mu^2}g_1(t_k^\ell)\geq a$ whenever $0\leq t-t_k^\ell< \mu^2/\varepsilon$.
On the other hand, the last conclusion in property (b) implies that
$|\{s\in \{t_k^\ell,\dots,t_k^r+1\}:\;g_1(s)\leq a\}|> \delta (t_k^r+1-t_k^\ell+1)>\delta |I_k'|$, as $t_k^r+1\in\widetilde I$.

\item[(d)] Property (c), in its turn, implies that for any $k\leq h$ we have $1\leq \delta |I_k'|/2$, whence
$|\{t\in I'_k:\;g_1(t)\leq a\}|\geq \delta |I'_k|/2$.

\end{itemize}

Our goal is to apply Lemma~\ref{l: aux averaging} to the constructed intervals.
For each $k\leq h$, we define the partition $I'_k=I'_{k,1}\cup I'_{k,2}\cup I'_{k,3}$,
where
$$I'_{k,1}:=\big\{t\in I'_k:\;g_1(t)\geq 2^{\mu^2} a\big\};\quad I_{k,3}':=
\big\{t\in I'_k:\;g_1(t)\leq a\big\};\quad I'_{k,2}:=I'_k\setminus(I_{k,1}'\cup I_{k,3}').$$
Additionally, set $\kappa:=\big(2^{\mu^2}-1\big)\cdot 4R$.
We define subset of {\it good} indices $G\subset[h]$ as
$$G:=\big\{k\leq h:\;|I'_{k,2}|\leq \delta |I'_k|\big\}.$$
Note that \eqref{eq: aux 30872976598}, together with property (a) of the intervals, implies that
$$\sum\limits_{k\in G}|I'_k|\geq \mu N-\sum\limits_{k\in [h]\setminus G}|I'_k|\geq \mu N-
\frac{1}{\delta}\frac{2N}{\lfloor 1/\mu^2\rfloor}\geq \mu N/2.$$
By Lemma~\ref{l: aux averaging}, for every $k\in G$ the event
$$\Event_k:=\big\{\big|\big\{t\in I'_k:\;|g_1(t)-g_2(t+Y)|\geq\kappa/2\big\}\big|<\delta |I'_k|/4\big\}$$
has probability at most $64\delta$. 
Hence, the expectation of the sum
$$\sum\limits_{k\in G}|I'_k|{\bf 1}_{\Event_k}$$
is at most $64\delta \cdot \sum\limits_{k\in G}|I'_k|$, and in view of Markov's inequality and the
lower bound for $\sum\limits_{k\in G}|I'_k|$,
$$\Prob\Big\{\sum\limits_{k\in G}|I'_k|{\bf 1}_{\Event_k^c}
\geq \frac{\mu N}{4}\Big\}
=1-\Prob\Big\{\sum\limits_{k\in G}|I'_k|{\bf 1}_{\Event_k}
> \sum\limits_{k\in G}|I'_k|-\frac{\mu N}{4}\Big\}
\geq 1-\frac{64\delta \, \sum_{k\in G}|I'_k|}{\frac{1}{2}\sum_{k\in G}|I'_k|}
= 1-128\delta.$$

As the final remark, for any realization of $Y$ such that $\sum\limits_{k\in G}|I'_k|{\bf 1}_{\Event_k^c}
\geq \frac{\mu N}{4}$, we have
$\big|\big\{t\in \widetilde I:\;|g_1(t)-g_2(t+Y)|\geq\kappa/2\big\}\big|\geq \frac{\delta}{4}\frac{\mu N}{4}$,
whence, in view of Lemma~\ref{l: aux -98745220985-32}
$$\mbox{$\big\|(1-p)\,g_1(\cdot)+p\,g_2(\cdot+Y)\big\|_2^2
\leq \big((1-p)\,\|g_1\|_2^2+p\,\|g_2\|_2^2\big)-p(1-p)\,\frac{\kappa^2}{4}\frac{\mu N}{4}\frac{\delta}{4}$}.$$
The result follows.
\end{proof}

The estimate on the $\ell_2$--norm
of the average in Proposition~\ref{prop: ell 2 update}
involves the parameter $\mu$ which, roughly speaking, determines the cardinality of the largest cluster of spikes in $g_1$.
If the cardinality is small, the estimate given by the proposition becomes weaker.
Even assuming best possible values for $\mu$, $n$ applications
of the averaging to obtain $f_{\mathcal A,p,n}$ from $f$
would not provide a bound on $\|f_{\mathcal A,p,n}\|_2$ which could be translated
into a meaningful estimate for the $\ell_\infty$--norm of the average.

Returning to the example that we discussed on page~\pageref{pge: desc 1}, if the function
$g_{N,d,I,\eta}$ is such that $|I|$ is much less than $d$, i.e.\ the spikes are rare
then with probability $1-\Theta(\frac{|I|}{d})\approx 1$ the averaged function $g^{av}_{N,d,I,\eta}$
will not have {\it any} spikes left.
When the spikes are located in an irregular fashion, such strong property does not hold, but the following phenomenon can
still be observed: if the spikes are rare then with a probability close to one the averaged function will
have much fewer (by a large factor) spikes.
In other words, in the regime when there are few points where the function is large,
rather than measuring the $\ell_2$--norm of the average, it is more useful to consider how the cardinality of
the set of spikes shrinks under averaging. Combining this idea with Proposition~\ref{prop: ell 2 update}, we can derive the following statement:

\begin{prop}\label{prop: refinement}
For any $p\in(0,1/2]$, $\varepsilon\in(0,1)$, $\widetilde R\geq 1$, $L_0\geq 16 \widetilde R$ and $M\geq 1$ there are
$n_{\text{\tiny\ref{prop: refinement}}}=n_{\text{\tiny\ref{prop: refinement}}}(p,\varepsilon,L_0,\widetilde R,M)>0$
and $\eta_{\text{\tiny\ref{prop: refinement}}}=\eta_{\text{\tiny\ref{prop: refinement}}}(p,\varepsilon,L_0,\widetilde R,M)\in(0,1)$
with the following property. Let $L_0\geq L\geq 16\widetilde R$, let $n\geq n_{\text{\tiny\ref{prop: refinement}}}$, $N\leq 2^n$,
let $g\in\ell_1(\Z)$ be a non-negative function satisfying
\begin{itemize}
\item $\|g\|_1=1$;
\item $\log_2 g$ is $\eta_{\text{\tiny\ref{prop: refinement}}}$--Lipschitz;
\item $\sum\limits_{t\in I}g(t)\leq \frac{\widetilde R}{\sqrt{n}}$ for any integer interval $I$ of cardinality $N$;
\item $\|g\|_\infty\leq \frac{L}{N\sqrt{n}}$.
\end{itemize}
For each $i\leq \lfloor \varepsilon n\rfloor$, let $X_i$ be a random variable uniform on some disjoint
union of integer intervals of cardinality at least $N$ each;
and assume that $X_1,\dots,X_{\lfloor \varepsilon n\rfloor}$ are independent.
Define a random function $\widetilde g\in\ell_1(\Z)$ as
$$\widetilde g(t):=\Exp_b\, g\Big(t+\sum_{i=1}^{\lfloor \varepsilon n\rfloor}b_i X_i\Big)
=\sum\limits_{(v_i)_{i=1}^{\lfloor \varepsilon n\rfloor}\in\{0;1\}^{\lfloor \varepsilon n\rfloor}}
p^{\sum_i v_i}(1-p)^{\lfloor \varepsilon n\rfloor-\sum_i v_i}\,g\big(t+v_1 X_1
+\dots+v_{\lfloor \varepsilon n\rfloor}X_{\lfloor \varepsilon n\rfloor}\big),$$
where $b=(b_1,\dots,b_n)$ is the vector of independent Bernoulli($p$) components.
Then
$$\Prob\big\{\mbox{$\|\widetilde g\|_\infty> \frac{(p/\sqrt{2}+1-p)L}{N\sqrt{n}}$}\big\}\leq \exp(-M n).$$
\end{prop}

In words, the above proposition tells us that, given a ``preprocessed'' function $g$ with
$\|g\|_\infty\leq \frac{L}{N\sqrt{n}}$,
after $\varepsilon n$ averagings the $\ell_\infty$--norm of the function drops at least by the factor $p/\sqrt{2}+1-p$
with a probability superexponentially close to one.
By applying the proposition several times to a ``preprocessed'' function given by Proposition~\ref{p: rough decay},
we will be able to complete the proof of the theorem.

Before proving the proposition, let us consider a simple lemma.

\begin{lemma}\label{l: aux 07395875092835}
Let $f\in\ell_1(\Z)$ be a non-negative function, let $m,N\in\N$, $p\in(0,1)$, $H,\mu>0$,
and assume that $\|f\|_\infty\leq 2H$ and that for any integer interval $I$ of cardinality $N$ we have
$$\big|\big\{t\in I:\;f(t)\geq H\big\}\big|\leq \mu N.$$
Choose any integers $x_1,x_2,\dots,x_m$ and set
$$\widetilde f(t):=\Exp_b \,f\big(t+b_1 x_1
+\dots+b_{m}x_{m}\big),
%\sum\limits_{(v_i)_{i=1}^{m}\in\{0;1\}^{m}}
%p^{\sum_i v_i}(1-p)^{m-\sum_i v_i}\,,\quad t\in\Z.
$$
where $b=(b_1,\dots,b_m)$ is the vector of independent Bernoulli($p$) random variables.
Then for any integer interval $J$ of cardinality $N$ we have
$$
\big|\big\{t\in J:\;\widetilde f(t)\geq \sqrt{2}H\big\}\big|\leq \mu N/\big(\sqrt{2}-1\big).
$$ 
\end{lemma}
\begin{proof}
Take any point $t\in\Z$ such that $\widetilde f(t)\geq \sqrt{2}H$.
We have
\begin{align*}
\sqrt{2}H\leq \widetilde f(t)&\leq %\sum\limits_{(v_i)_{i=1}^{m}\in\{0;1\}^{m}}
%p^{\sum_i v_i}(1-p)^{m-\sum_i v_i}\,{\bf 1}_{\{f(\cdot+v_1 x_1
%+\dots+v_{m}x_{m})\geq H\}}(t)\cdot 2H\\
\Exp_b\,{\bf 1}_{\{f(\cdot+b_1 x_1
+\dots+b_{m}x_{m})\geq H\}}(t)\cdot 2H
+
\Exp_b\,{\bf 1}_{\{f(\cdot+b_1 x_1+\dots+b_{m}x_{m})< H\}}(t)\cdot H\\
%\sum\limits_{(v_i)_{i=1}^{m}\in\{0;1\}^{m}}
%p^{\sum_i v_i}(1-p)^{m-\sum_i v_i}\,{\bf 1}_{\{f(\cdot+v_1 y_1
%+\dots+v_{m}y_{m})< H\}}(t)\cdot H\\
&=H+H\,\Exp_b\,{\bf 1}_{\{f(\cdot+b_1 x_1
+\dots+b_{m}x_{m})\geq H\}}(t),
%\sum\limits_{(v_i)_{i=1}^{m}\in\{0;1\}^{m}}
%p^{\sum_i v_i}(1-p)^{m-\sum_i v_i}\,{\bf 1}_{\{f(\cdot+v_1 y_1
%+\dots+v_{m}y_{m})\geq H\}}(t),
\end{align*}
so that
\begin{equation}\label{eq: aux 3998762453646}
%\sum\limits_{(v_i)_{i=1}^{m}\in\{0;1\}^{m}}
%p^{\sum_i v_i}(1-p)^{m-\sum_i v_i}\,{\bf 1}_{\{f(\cdot+v_1 y_1
%+\dots+v_{m}y_{m})\geq H\}}(t)
\Exp_b\,{\bf 1}_{\{f(\cdot+b_1 x_1
+\dots+b_{m}x_{m})\geq H\}}(t)
\geq \sqrt{2}-1.
\end{equation}
On the other hand, for any interval $J$ of cardinality $N$ and any choice of
$(v_i)_{i=1}^{m}\in\{0,1\}^m$, we have, by the assumptions of the lemma,
$$
\sum\limits_{s\in J}{\bf 1}_{\{f(\cdot+v_1 x_1
+\dots+v_{m}x_{m})\geq H\}}(s)\leq \mu N,
$$
whence
$$
\sum\limits_{s\in J}%\sum\limits_{(v_i)_{i=1}^{m}\in\{0;1\}^{m}}
%p^{\sum_i v_i}(1-p)^{m-\sum_i v_i}\,{\bf 1}_{\{f(\cdot+v_1 y_1
%+\dots+v_{m}y_{m})\geq H\}}(s)
\Exp_b\,{\bf 1}_{\{f(\cdot+b_1 y_1
+\dots+b_{m}y_{m})\geq H\}}(s)
\leq \mu N.
$$
Combining the last inequality with the condition \eqref{eq: aux 3998762453646}, we get the statement.
\end{proof}

\begin{proof}[{Proof of Proposition~\ref{prop: refinement}}]
Fix any admissible parameters $\varepsilon$, $p$, $\widetilde R$, $L$, $N$ and $M$, and set
$$\mu:=\frac{1}{24 C_{\text{\tiny\ref{prop: ell 2 update}}}}\exp\bigg(-\frac{16M}{\varepsilon}\bigg);\quad
\eta:=\mu^4.$$
We will assume that $n$ is sufficiently large so that $\varepsilon n/4\geq 1$ and, moreover,
\begin{equation}\label{eq: aux 209823059325-2}
c_{\text{\tiny\ref{prop: ell 2 update}}} p(1-p) \lfloor \varepsilon n/2\rfloor \mu^6 \widetilde R^2/2> L_0\sqrt{n}.
\end{equation}
Set
$$m:=\lfloor \varepsilon n/2\rfloor;\quad H:=\frac{L}{2N\sqrt{n}}.$$
We fix any function $g\in\ell_1(\Z)$ satisfying conditions of the proposition with parameters $\eta$, $\widetilde R$, $N$, $L$, $n$.
Note that $\|g\|_\infty\leq 2H$.
Define $g_0:=g$,
$$g_k(t):=p\,g_{k-1}(t+X_k)+(1-p)\,g_{k-1}(t),\quad k=1,2,\dots,\lfloor \varepsilon n\rfloor,\quad t\in\Z,$$
so that either $\widetilde g=g_{2m}$ (if $\lfloor \varepsilon n\rfloor$ is even) or $\widetilde g=g_{2m+1}$ (if $\lfloor \varepsilon n\rfloor$ is odd).
It is easy to see that $\log_2 g_k$ is $\eta$--Lipschitz (because the log-Lipschitzness
is preserved under taking convex combinations) and $\|g_k\|_1=1$ for all admissible $k$.

For each $i\leq m$, define events
$$\Event_i:=
\Big\{\big|\big\{t\in I:\;g_i(t)\geq H\big\}\big|\leq \mu N\quad \mbox{ for any integer interval $I$ of cardinality $N$}\Big\},
$$
and
$$\widetilde \Event_i:=
\Big\{\mbox{$\|g_i\|_2^2\leq \|g_{i-1}\|_2^2
-c_{\text{\tiny\ref{prop: ell 2 update}}} p(1-p)\mu^6 \widetilde R^2 n^{-1}N^{-1}
$}\Big\}
$$
(we can formally extend the first definition to $i=0$).
Clearly, for each $i$, $\Event_i$ and $\widetilde \Event_i$
are measurable w.r.t the sigma-algebra generated by $X_1,\dots,X_i$.
Condition for a moment on any realization of $X_1,\dots,X_{i-1}$, and observe that one of the following two assertions is true:
\begin{itemize}
\item $\Event_{i-1}$ holds;
\item $\big|\big\{t\in I:\;g_i(t)\geq 8R\big\}\big|\geq \mu N$ for some integer interval $I$ of cardinality $N$,
where we set $R:=\frac{\widetilde R}{N\sqrt{n}}$.
Then, applying Proposition~\ref{prop: ell 2 update}, we get
$\Prob_{X_i}(\widetilde \Event_i)\geq 1-C_{\text{\tiny\ref{prop: ell 2 update}}}\mu$.
\end{itemize}
Hence,
$$\Prob\big(\Event_{i-1}\cup\widetilde \Event_i\;\vert\;X_1,\dots,X_{i-1}\big)\geq 1-C_{\text{\tiny\ref{prop: ell 2 update}}}\mu,
\quad i=1,2,\dots,m.$$
This implies that for any $r\in[m]$, the probability that $\big(\Event_{i-1}\cup\widetilde \Event_i\big)^c$
holds for at least $r$ indices $i$ can be estimated as
$$\Prob\Big(
\bigcup_{S\subset[m],|S|=r}\;\bigcap_{i\in S}\big(\Event_{i-1}\cup\widetilde \Event_i\big)^c\Big)
\leq {m\choose r}\big(C_{\text{\tiny\ref{prop: ell 2 update}}}\mu\big)^r.
$$
Note that the definition of $g_k$'s and the triangle inequality imply that the sequence $\big(\|g_k\|_2\big)_{k\geq 0}$
is non-increasing. Hence, taking $r:=\lceil m/2\rceil$ in the above formula
and in view of our choice of $\mu$, we get that with probability at least $1-\exp(-2M n)$
at least one of the following two conditions is satisfied:
\begin{itemize}
\item[(a)] 
There is $i\leq m$ such that
$\big|\big\{t\in I:\;g_i(t)\geq H\big\}\big|\leq \mu N$ for any integer interval $I$ of cardinality $N$; or
\item[(b)]
$\|g_m\|_2^2\leq \|g\|_2^2
-c_{\text{\tiny\ref{prop: ell 2 update}}} p(1-p) m\mu^6 \widetilde R^2 n^{-1}N^{-1}/2
$.
\end{itemize}
It can be checked, however, that condition (b) is improbable.
Indeed, in view of the restrictions on the $\ell_1$-- and $\ell_\infty$--norms of $g$,
and H\" older's inequality,
$$
\|g\|_2^2\leq 1\cdot \frac{L}{N\sqrt{n}},
$$
whence, applying \eqref{eq: aux 209823059325-2}, we get
$\|g\|_2^2-c_{\text{\tiny\ref{prop: ell 2 update}}} p(1-p) m\mu^6 \widetilde R^2 n^{-1}N^{-1}/2<0$.

Thus, only (a) may hold, so the event
\begin{align*}
\Event:=\big\{
&\mbox{There is $i\leq m$ such that
$\big|\big\{t\in I:\;g_i(t)\geq H\big\}\big|\leq \mu N$}\\
&\mbox{for any integer interval $I$ of cardinality $N$}
\big\}
\end{align*}
has probability at least $1-\exp(-2Mn)$.
Applying Lemma~\ref{l: aux 07395875092835}
we get that everywhere on the event
\begin{equation}\label{eq: aux 09502875605824}
\big|\big\{t\in I:\;g_i(t)\geq \sqrt{2}H\big\}\big|\leq 3\mu N\mbox{ for any interval $I$ of cardinality $N$ and $i\geq m+1$}.
\end{equation}

\medskip

The second part of our proof resembles the proof of Proposition~\ref{p: rough decay},
although the argument here is simpler.
We observe that there exists a random sequence of integers $(t_{i})_{i=m}^{2m}$
satisfying
\begin{itemize}
\item The sequence $\big(g_i(t_i)\big)_{i=m}^{2m}$ is non-increasing;
\item $g_{2m}(t_{2m})=\|g_{2m}\|_\infty$;
\item $t_i\in\{t_{i-1},t_{i-1}-X_i\}$ for all $m<i\leq 2m$.
\end{itemize}
On the event
$$\hat\Event:=\big\{\|g_{\lfloor \varepsilon n\rfloor}\|_\infty\geq (\sqrt{2}p+2(1-p))H\big\}$$
we necessarily have $\|g_{i}\|_\infty\geq (\sqrt{2}p+2(1-p))H$, $i\leq 2m$,
hence, in view of the recursive relation
$g_i(t_i)=p\,g_{i-1}(t_i+X_i)+(1-p)g_{i-1}(t_i)$ and the deterministic upper bound $\|g_{i-1}\|_\infty\leq 2H$,
we have
$g_{i-1}(t_i+X_i)\geq \sqrt{2}H$ and $g_{i-1}(t_i)\geq \sqrt{2}H$ for all $m<i\leq 2m$.
Thus,
$$\hat\Event \subset\big\{
\mbox{$g_{i-1}(t_i+X_i)\geq \sqrt{2}H$ and $g_{i-1}(t_i)\geq \sqrt{2}H$ for all $m<i\leq 2m$}
\big\}.$$
We will show that the probability of the latter event is small by considering a union bound over non-random
sequences.

Fix any realizations $X_1^0,\dots, X_m^0$ of $X_1,\dots,X_m$ such that the event $\Event$ defined above holds.
Take any non-random sequence $(v_i)_{i=m+1}^{2m}\in\{0,1\}^m$ and any fixed $s_m\in\Z$
such that $g_m(s_m)\geq \sqrt{2}H$ (if such $s_m$ exists). Further, we define random numbers $s_i:= s_{i-1}-v_i X_i$, $i=m+1,\dots,2m$.
Then for any $i\geq m+1$ we have
\begin{align*}
\Prob\big\{&\mbox{$g_{i-1}(s_i+X_i)\geq \sqrt{2}H$ and $g_{i-1}(s_i)\geq \sqrt{2}H$}\;\vert\;
X_1=X_1^0,\dots, X_m=X_m^0;X_{m+1},\dots,X_{i-1}\big\}\\
&=
\Prob\big\{\mbox{$g_{i-1}(s_{i-1}+(1-v_i)X_i)\geq \sqrt{2}H$ and }\\
&\hspace{1cm}\mbox{$g_{i-1}(s_{i-1}-v_i X_i)\geq \sqrt{2}H$}
\;\vert\;
X_1=X_1^0,\dots, X_m=X_m^0;X_{m+1},\dots,X_{i-1}\big\}\\
&\leq 
\Prob\big\{\mbox{$g_{i-1}(s_{i-1}+X_i)\geq \sqrt{2}H$ or }\\
&\hspace{1cm}\mbox{$g_{i-1}(s_{i-1}- X_i)\geq \sqrt{2}H$}
\;\vert\;
X_1=X_1^0,\dots, X_m=X_m^0;X_{m+1},\dots,X_{i-1}\big\}\\
&\leq 2\cdot 2\cdot 3\mu,
\end{align*}
in view of \eqref{eq: aux 09502875605824} and our assumption about the distribution of $X_i$'s.
Hence,
$$\Prob\big\{\mbox{$g_{i-1}(s_i+X_i)\geq \sqrt{2}H$ and $g_{i-1}(s_i)\geq \sqrt{2}H$ for all $m<i\leq 2m$}
\;\vert\;
X_1=X_1^0,\dots, X_m=X_m^0\big\}$$
is at most $(12\mu)^m$.
This, together with the obvious observation $|\{s\in\Z:\;g_m(s)\geq \sqrt{2}H\}|\leq (\sqrt{2}H)^{-1}$,
allows to estimate the probability of $\hat \Event$ as
$$\Prob(\hat\Event)\leq \Prob(\Event^c)+2^m (\sqrt{2}H)^{-1} (12\mu)^m\leq \exp(-2M n)+2^m (\sqrt{2}H)^{-1} (12\mu)^m.$$
By our definition of the parameters $\mu,H,m$, the rightmost quantity is less than $\exp(-Mn)$
for all sufficiently large $n$.
The proof is complete.
\end{proof}

\bigskip

\begin{proof}[Proof of Theorem~\ref{th: averaging}]
Fix any admissible parameters $\delta\in(0,1]$, $p\in(0,1/2]$, $\varepsilon\in(0,p)$, $K,M\geq 1$.
The proof of the theorem is essentially a combination of Proposition~\ref{p: rough decay}
which provides a rough bound on the $\ell_\infty$--norm which depends on $M$,
and subsequent application of Proposition~\ref{prop: refinement} to get
a refined bound.

We define
$$L:=L_{\text{\tiny\ref{p: rough decay}}}(2M,p,\delta,\varepsilon/2);\quad
\widetilde R:=\frac{C_{\text{\tiny\ref{l: aux simple ac}}}}{\sqrt{\min(\delta,1/2)\,p}},$$
and let $q$ be the smallest positive integer such that $\big(p/\sqrt{2}+1-p\big)^q \leq L^{-1}$.
Further, define $\alpha=\alpha(p,\varepsilon)$ as the smallest number in $[1/2,1)$ which satisfies
$$(1-p+\varepsilon)^{1-\alpha}\geq\bigg(\frac{1-p+\varepsilon/2}{1-p+\varepsilon}\bigg)^{1/4},$$
and set $\widetilde \varepsilon:=(1-\alpha)/(2q)$.
Now, we fix any $n$ satisfying
\begin{align*}
&\min(\delta,\widetilde \varepsilon,1/2)n\geq 1;
&n^{\frac{1}{2n}}\leq \bigg(\frac{1-p+\varepsilon}{1-p+\varepsilon/2}\bigg)^{1/4};\\
&n\geq n_{\text{\tiny\ref{prop: refinement}}}\big(p,\widetilde\varepsilon,\max(16\widetilde R,L),\widetilde R,2M\big);
&n\geq n_{\text{\tiny\ref{p: rough decay}}}(2M,p,\delta,\varepsilon/2),
\end{align*}
fix $1\leq N\leq (1-p+\varepsilon)^{-n}$,
and define $\ell:=\lceil \alpha n\rceil$. It can be checked that with the above assumptions on parameters,
we have $(1-p+\varepsilon/2)^{\ell}\leq (1-p+\varepsilon)^n/\sqrt{n}$.

Further, we fix any non-negative function $f\in\ell_1(\Z)$ with $\|f\|_1=1$ and such that
$\log_2 f$ is $\eta$--Lipschitz
for $\eta=\eta_{\text{\tiny\ref{prop: refinement}}}(p,\widetilde\varepsilon,\max(16\widetilde R,L),\widetilde R,2M)$.
Note that, by the above, $(1-p+\varepsilon/2)^{\ell}\,\|f\|_\infty\leq L(N\sqrt{n})^{-1}$,
and, by Proposition~\ref{p: rough decay}, the event
$$\Event_{\text{\tiny\ref{p: rough decay}}}:=\big\{\|f_{\mathcal A,p,\ell}\|_\infty\leq
L(N\sqrt{n})^{-1}\big\}$$
has probability at least $1-\exp(-2Mn)$.

Further, we split the integer interval $\{\ell,\ell+1,\dots,n\}$ into $q$ subintervals, each of cardinality
at least $\frac{n- \alpha n}{2q}= \widetilde \varepsilon n$.
Let $\ell\leq i_1<i_2<\dots<i_q=n$ be the right endpoints of corresponding subintervals.
Observe that by Lemma~\ref{l: aux simple ac}, for any $k\geq \ell$ and any integer
interval $I$ of cardinality $N$ we have deterministic relation
$$\sum\limits_{t\in I}f_{\mathcal A,p,k}(t)\leq \frac{C_{\text{\tiny\ref{l: aux simple ac}}}}{\sqrt{\min(\delta,1/2) n\,p}}
=\frac{\widetilde R}{\sqrt{n}},$$
by our definition of $R$.
This enables us to apply Proposition~\ref{prop: refinement}.
Applying Proposition~\ref{prop: refinement} to the first subinterval, we get that,
conditioned on the event $\Event_0:=\Event_{\text{\tiny\ref{p: rough decay}}}$, the event
$$\Event_1:=\Big\{\mbox{$\|f_{\mathcal A,p,i_1}\|_\infty\leq \frac{\max(16\widetilde R,(p/\sqrt{2}+1-p)L)}{N\sqrt{n}}$}\Big\}$$
has probability at least $1-\exp(-2Mn)$. More generally, for the $j$-th subinterval,
the application of Proposition~\ref{prop: refinement} gives
$$\Prob\big(\Event_j\;\vert\;\Event_{j-1}\big)\geq 1-\exp(-2Mn),$$
where for each $1\leq j\leq q$,
$$\Event_j:=\Big\{\mbox{$\|f_{\mathcal A,p,i_j}\|_\infty\leq \frac{\max(16\widetilde R,(p/\sqrt{2}+1-p)^j L)}{N\sqrt{n}}$}\Big\}.$$
Taking into account our definition of $q$,
$$\Event_q=\Big\{\mbox{$\|f_{\mathcal A,p,n}\|_\infty\leq \frac{16\widetilde R}{N\sqrt{n}}$}\Big\}.$$
In view of the above, the probability of this event can be estimated from below by $1-(q+1)\exp(-2Mn)$, which is greater than $1-\exp(-Mn)$
for all suffificently large $n$.
It remains to choose
$$L_{\text{\tiny\ref{th: averaging}}}:=16\widetilde R.$$
\end{proof}

\section{Proof of Theorem~A}

Let us recall the definition of a threshold which we considered in Section~\ref{s: strategy}.
For any $p\in(0,1/2]$, any vector $x\in S^{n-1}$ and any parameter $L>0$ we define the {\it threshold}
$\thres_p(x,L)$ as the supremum of all $t\in(0,1]$ such that
$\cf\big(\sum_{i=1}^n b_i x_i,t\big)> Lt$,
where $b_1,\dots,b_n$ are independent Bernoulli($p$) random variables.
Note that $\thres_p(x,L)\geq\frac{1}{L} (1-p)^n$.
On the other hand, as a consequence of the L\'evy--Kolmogorov--Rogozin
inequality (Lemma~\ref{l: lkr}), we obtain
\begin{lemma}\label{l: threshold}
For every $p\in(0,1/2]$, $\delta,\nu\in(0,1]$ there are $K_{\text{\tiny\ref{l: threshold}}}
=K_{\text{\tiny\ref{l: threshold}}}(p,\delta,\nu)>0$ and
$L_{\text{\tiny\ref{l: threshold}}}=L_{\text{\tiny\ref{l: threshold}}}(p,\delta,\nu)\geq 1$
with the following property. Let $n\geq 2$, $L\geq L_{\text{\tiny\ref{l: threshold}}}$,
and let $x\in\incomp_n(\delta,\nu)$. Then $\thres_p(x,L)\leq \frac{K_{\text{\tiny\ref{l: threshold}}}}{\sqrt{n}}$.
\end{lemma}
\begin{proof}
Take any vector $x\in\incomp_n(\delta,\nu)$, and let $I\subset[n]$ be a
subset of cardinality $\lfloor \delta n\rfloor$ corresponding to the largest (by absolute value) coordinates of
$x$, i.e.\ such that $|x_i|\geq |x_\ell|$ for all $i\in I$ and $\ell\in[n]\setminus I$.
Since $x$ is $(\delta,\nu)$--incompressible, we have
$\|x\,{\bf 1}_{[n]\setminus I}\|_2\geq \nu$, whence there is $\ell\in [n]\setminus I$
such that $|x_\ell|\geq \nu/\sqrt{n}$. Thus, $|x_i|\geq \nu/\sqrt{n}$ for all $i\in I$.
Applying Lemma~\ref{l: lkr}, we get
$$\cf\Big(\sum\limits_{i= 1}^n b_i x_i,\mbox{$\frac{\nu t}{\sqrt{n}}$}\Big)
\leq \cf\Big(\sum\limits_{i\in I}b_i x_i,\mbox{$\frac{\nu t}{\sqrt{n}}$}\Big)=
\cf\Big(\mbox{$\frac{\sqrt{n}}{\nu}$}\sum\limits_{i\in I}b_i x_i,t\Big)\leq \frac{Ct}{\sqrt{\lfloor \delta n\rfloor}}$$
for all $t\geq 1$ for some $C\geq 1$ depending only on $p$.
It remains to choose $L_{\text{\tiny\ref{l: threshold}}}:=\frac{C}{\nu\sqrt{\delta/2}}$
and $K_{\text{\tiny\ref{l: threshold}}}:=\max\big(\delta^{-1/2},\nu\big)$.
The result follows by the definition of the threshold.
\end{proof}
\begin{Remark}
The above lemma can also be obtained by applying results of \cite{RV adv},
namely, the property that the least common denominator of an incompressible vector is of order at least $\sqrt{n}$.
\end{Remark}

\medskip

Let us discuss what is left in order to complete the proof of Theorem~A.
The standard decomposition of $S^{n-1}$ into sets of compressible and incompressible vectors
and the reduction of invertibility over the incompressible vectors to the distance problem
for the random normal (see description in Section~\ref{s: strategy}),
leave the following question: given a number $T\gg (1-p+\varepsilon)^n$,
show that the probability of the event $\{\thres_p(Y_n,L)\in [T,2T)\}$ is close to zero.
Here, $Y_n$ is a unit normal vector to the first $n-1$ columns of the matrix $B_n(p)+s\,1_n1_n^\top$.
Assuming that $\Net_T$ is a discrete approximation of the set of incompressible vectors with the threshold in $[T,2T)$,
we can write
$$
\Prob\big\{\thres_p(Y_n,L)\in [T,2T)\big\}
\leq |\Net_T|\,\max\limits_{x\in\Net_T}\Prob\big\{\mbox{$x$ is ``almost orthogonal'' to $\Col_1,\dots,\Col_{n-1}$}\big\}
$$
(we prefer not to specify at this stage what ``almost orthogonal'' means quantitatively).
Most of the work related to estimating the cardinality of $\Net_T$ was done in Section~\ref{s: averaging}.
Here, we combine Corollary~\ref{cor: anticoncentration} with a simple counting argument
giving an estimate of the cardinality of a part of the integer lattice $\Z^n$ with prescribed bounds on
the vector coordinates (see Corollary~\ref{cor: permutations} in this section).
The probability estimate for the event
$$
\big\{\mbox{$x$ is ``almost orthogonal'' to $\Col_1,\dots,\Col_{n-1}$}\big\}
$$
would follow as a simple consequence of the Tensorization Lemma~\ref{l: tensorization}
and individual small ball probability bounds for $\langle x,\Col_i\rangle$.
Note that if the threshold of the vector $x$ was contained in the range $[0,C\,T)$,
such estimates would immediately follow from the definition of the threshold. However, the vector $x\in\Net_T$
is only an approximation of another vector with a small threshold. Thus, to make the conclusion,
we will need a statement which asserts that for a given vector one can find its lattice approximation which
preserves (to some extent) the anticoncentration properties of the corresponding random linear combination:

\begin{lemma}\label{l: magic vector}
Let $p\in(0,1/2]$, let $y=(y_1,\dots,y_n)\in\R^n$ be a vector and $L>0$, $\lambda\in\R$ be numbers such that
for mutually independent Bernoulli($p$) random variables $b_1,\dots,b_n$ we have
$\Prob\{\big|\sum_{i=1}^n b_i y_i-\lambda\big|\leq t\}\leq Lt$ for all $t\geq\sqrt{n}$.
Then there exists a vector $y'=(y_1',\dots,y_n')\in\Z^n$ having the following properties
\begin{itemize}
\item $\|y-y'\|_\infty\leq 1$;
\item
$\Prob\big\{\big|\sum_{i=1}^n b_i y_i'-\lambda\big|\leq t\big\}
\leq C_{\text{\tiny\ref{l: magic vector}}}\,L t$ for all $t\geq \sqrt{n}$;
\item $\cf\big(\sum_{i=1}^n b_i y_i',\sqrt{n}\big)\geq c_{\text{\tiny\ref{l: magic vector}}}\,\cf\big(\sum_{i=1}^n b_i y_i,\sqrt{n}\big)$;
\item $\big|\sum_{i=1}^n y_i-\sum_{i=1}^n y_i'\big|\leq C_{\text{\tiny\ref{l: magic vector}}}\sqrt{n}$.
\end{itemize}
Here,
$C_{\text{\tiny\ref{l: magic vector}}},c_{\text{\tiny\ref{l: magic vector}}}>0$
are universal constants.
\end{lemma}

The first and the last property of $y'$ will be used to estimate the Euclidean norm of
$(B_n(p)+s\,1_n1_n^\top)(y-y')$:
the bound on $\|y-y'\|_\infty$ provides control of $\|(B_n(p)-p\,1_n1_n^\top)(y-y')\|_2$ while
the relation $\big|\sum_{i=1}^n y_i-\sum_{i=1}^n y_i'\big|\leq C_{\text{\tiny\ref{l: magic vector}}}\sqrt{n}$
implies $\big\|(s+p)\,1_n1_n^\top(y-y')\big\|_2\leq C_{\text{\tiny\ref{l: magic vector}}}|s+p|n$.

The proof of Lemma~\ref{l: magic vector} is based on a well known concept of the {\it randomized rounding} \cite{RRounding}
(see also \cite{AK,KlLiv,Livshyts} for some recent applications). 
%which originated in computer science and has recently been applied in several papers (see \cite{AK,KlLiv,Livshyts}).
The first use of this method in the context of matrix invertibility is, to the best of author's knowledge,
due to G.Livshyts \cite{Livshyts}. In \cite{Livshyts}, the randomized rounding is used to choose
a best lattice approximation for a vector, which in turn is applied to construction of $\varepsilon$--nets;
our work follows the same principle.
We note that, unlike \cite{Livshyts}, in the present paper we need to explicitly control the L\'evy
concentration function and the small ball probability estimates for the approximating vector
(the second and the third property in the statement).

\begin{proof}[Proof of Lemma~\ref{l: magic vector}]
Fix a vector $y\in\R^n$, and let $b_1,\dots,b_n$ be independent Bernoulli($p$) random variables.
Further, let $\xi_1,\dots,\xi_n$ be random variables jointly independent with $b_1,\dots,b_n$,
such that for each $i\leq n$, $\xi_i$ takes values $\lfloor y_i\rfloor$ and $\lfloor y_i\rfloor+1$
with probabilities $\lfloor y_i\rfloor+1-y_i$ and $y_i-\lfloor y_i\rfloor$, respectively
(so that $\Exp\,\xi_i=y_i$).
Define random vector $\widetilde y:=(\xi_1,\dots,\xi_n)$, and observe that with probability one
$\|y-\widetilde y\|_\infty\leq 1$.

Fix for a moment any $w>0$ and denote by $S(2w)$ the collection of all
$(v_i)_{i=1}^n\in\{0,1\}^n$ such that $\big|\sum_{i=1}^n v_i y_i-\lambda\big|> 2w$.
Take any $(v_i)_{i=1}^n\in S(2w)$.
Note that $\sum_{i=1}^n v_i (y_i-\widetilde y_i)$ is the sum of independent variables, each of mean zero
and variance at most $1/4$.
Hence, by Markov's inequality,
$$\Prob\Big\{\Big|\sum_{i=1}^n v_i \widetilde y_i-\lambda\Big|\leq w\Big\}
\leq \Prob\Big\{\Big|\sum_{i=1}^n v_i (y_i-\widetilde y_i)\Big|> w\Big\}
\leq \frac{n}{4w^2}.$$
Thus, if $\widetilde S(w)$ is the (random) collection of all vectors $(v_i)_{i=1}^n\in\{0,1\}^n$
such that $\big|\sum_{i=1}^n v_i \widetilde y_i-\lambda\big|> w$ then the above estimate immediately implies
for an arbitrary subset $E\subset\{0,1\}^n$:
\begin{align*}
\Exp \sum\limits_{(v_i)_{i=1}^n\in (S(2w)\setminus E)\setminus \widetilde S(w)}p^{\sum_i v_i}(1-p)^{n-\sum_i v_i}
&=\Exp 
\,\Exp_b\,{\bf 1}_{\{(b_i)_{i=1}^n\in (S(2w)\setminus E)\setminus \widetilde S(w)\}}\\
%\sum\limits_{(v_i)_{i=1}^n\in \{0,1\}^n}p^{\sum_i v_i}(1-p)^{n-\sum_i v_i}
%\,{\bf 1}_{\{(v_i)_{i=1}^n\in (S(2w)\setminus E)\setminus \widetilde S(w)\}}\\
&\leq \frac{n}{4w^2}\, %\sum\limits_{(v_i)_{i=1}^n\in \{0,1\}^n}p^{\sum_i v_i}(1-p)^{n-\sum_i v_i}
\Exp_b\,{\bf 1}_{\{(b_i)_{i=1}^n\in S(2w)\setminus E\}}\\
&= \frac{n}{4w^2} \sum\limits_{(v_i)_{i=1}^n\in S(2w)\setminus E}p^{\sum_i v_i}(1-p)^{n-\sum_i v_i}.
\end{align*}
%$$\Exp \big|(S(2t)\setminus E)\setminus \widetilde S(t)\big|\leq \frac{n}{4t^2}|S(2t)\setminus E|$$
We take $E=S(4w)$ in the above relation
and apply it for $w=2^{j-1} t$, $j\geq 1$, so that
\begin{align*}
\Exp \sum\limits_{(v_i)_{i=1}^n\in S(2t)\setminus \widetilde S(t)}
p^{\sum_i v_i}(1-p)^{n-\sum_i v_i}
&=
\Exp\,\sum\limits_{j=1}^\infty\sum\limits_{(v_i)_{i=1}^n\in
(S(2^j t)\setminus S(2^{j+1}t))\setminus \widetilde S(t)}
p^{\sum_i v_i}(1-p)^{n-\sum_i v_i}\\
&\leq
\Exp\,\sum\limits_{j=1}^\infty\sum\limits_{(v_i)_{i=1}^n\in
(S(2^j t)\setminus S(2^{j+1}t))\setminus \widetilde S(2^{j-1}t)}
p^{\sum_i v_i}(1-p)^{n-\sum_i v_i}\\
&\leq \sum\limits_{j=1}^\infty\frac{n}{2^{2j} t^2}
 \sum\limits_{(v_i)_{i=1}^n\in S(2^j t)\setminus S(2^{j+1}t)}p^{\sum_i v_i}(1-p)^{n-\sum_i v_i}\\
&\leq \sum\limits_{j=1}^\infty \frac{n\,L\,2^{j+1}t}{2^{2j} t^2}\\
&= \frac{2L\,n}{t}
\end{align*}
for any $t\geq \sqrt{n}$,
where we have used that, by the assumption on $y$,
$$
\sum\limits_{(v_i)_{i=1}^n\in
S(2^j t)\setminus S(2^{j+1}t)}
p^{\sum_i v_i}(1-p)^{n-\sum_i v_i}
\leq \Prob\Big\{\Big|\sum_{i=1}^n b_i y_i-\lambda\Big|\leq 2^{j+1}t\Big\}\leq
L\,2^{j+1}t.$$
The relation implies that for all $t\geq \sqrt{n}$,
\begin{align*}
\Exp &\max\bigg(0,\sum\limits_{(v_i)_{i=1}^n\in \{0,1\}^n\setminus \widetilde S(t)}
p^{\sum_i v_i}(1-p)^{n-\sum_i v_i}-
\sum\limits_{(v_i)_{i=1}^n\in \{0,1\}^n\setminus S(2t)}
p^{\sum_i v_i}(1-p)^{n-\sum_i v_i}\bigg)\\
&\leq \frac{2L\,n}{t}.
\end{align*}
An application of Markov's inequality, with $t=\sqrt{n},2\sqrt{n},4\sqrt{n},\dots$,
gives
\begin{align*}
\Prob\Big\{&\mbox{There exists integer $k\geq 0$ such that}\\
&\sum\limits_{(v_i)_{i=1}^n\in \{0,1\}^n}
p^{\sum_i v_i}(1-p)^{n-\sum_i v_i}{\bf 1}_{\{
|\sum_{i=1}^n v_i \widetilde y_i-\lambda|\leq 2^k\sqrt{n}
\}}
\geq 2^{3}L\,2^k\sqrt{n}\\
&\hspace{4cm}+\sum\limits_{(v_i)_{i=1}^n\in \{0,1\}^n\setminus S(2^{1+k}\sqrt{n})}
p^{\sum_i v_i}(1-p)^{n-\sum_i v_i}
\Big\}\leq \frac{1}{4}\sum\limits_{k=0}^\infty 2^{-2k}<\frac{7}{16}.
\end{align*}
Together with the condition on the small ball probability of random sums $\sum_{i=1}^n b_i y_i-\lambda$,
this implies that there is an event $\Event_1$ measurable with respect to $\widetilde y$ and with $\Prob(\Event_1)>9/16$
such that for any realization $\widetilde y^0$ of $\widetilde y$ from $\Event_1$,
$$\Prob\Big\{\Big|\sum_{i=1}^n b_i \widetilde y_i-\lambda\Big|\leq t\;\big\vert\;\widetilde y=\widetilde y^0\Big\}
\leq CL\quad\mbox{for all $t\geq\sqrt{n}$},$$
for some universal constant $C>0$.

Further, we will derive lower bounds on the anticoncentration function of the sum $\sum_{i=1}^n b_i \widetilde y_i$.
The argument is very similar to the one above, and we will skip some details.
Let $\lambda'\in\R$ be a number such that
$$\cf\Big(\sum_{i=1}^n b_i y_i,\sqrt{n}\Big)=
\sum\limits_{(v_i)_{i=1}^n\in \{0,1\}^n\setminus S_{\lambda'}(\sqrt{n})}
p^{\sum_i v_i}(1-p)^{n-\sum_i v_i},$$
where
$$S_{\lambda'}(\sqrt{n}):=\Big\{(v_i)_{i=1}^n\in\{0,1\}^n:\;
\Big|\sum_{i=1}^n v_i y_i-\lambda'\Big|> \sqrt{n}\Big\}.
$$
Further, denote
$$\widetilde S_{\lambda'}(2\sqrt{n}):=\Big\{(v_i)_{i=1}^n\in \{0,1\}^n:\;
\Big|\sum_{i=1}^n v_i \widetilde y_i-\lambda'\Big|> 2\sqrt{n}\Big\}.$$
Take any $(v_i)_{i=1}^n\in \{0,1\}^n\setminus S_{\lambda'}(\sqrt{n})$.
Since the variance of the random sum $\sum_{i=1}^n v_i (y_i-\widetilde y_i)$ is at most $n/4$,
we get
\begin{align*}
\Prob\Big\{\Big|\sum_{i=1}^n v_i \widetilde y_i-\lambda'\Big|> 2\sqrt{n}\Big\}
\leq \Prob\Big\{\Big|\sum_{i=1}^n v_i (y_i-\widetilde y_i)\Big|>\sqrt{n}\Big\}
\leq \frac{1}{4}.
\end{align*}
Hence,
$$
\Exp\,\sum\limits_{(v_i)\in (\{0,1\}^n\setminus S_{\lambda'}(\sqrt{n}))\cap \widetilde S_{\lambda'}(2\sqrt{n})}
p^{\sum_i v_i}(1-p)^{n-\sum_i v_i}
\leq
\frac{1}{4}
\sum\limits_{(v_i)\in \{0,1\}^n\setminus S_{\lambda'}(\sqrt{n})}
p^{\sum_i v_i}(1-p)^{n-\sum_i v_i},$$
so that with probability at least $2/3$ we have
\begin{equation}\label{eq: aux 2074450573}
\sum\limits_{(v_i)\in (\{0,1\}^n\setminus S_{\lambda'}(\sqrt{n}))\cap \widetilde S_{\lambda'}(2\sqrt{n})}
p^{\sum_i v_i}(1-p)^{n-\sum_i v_i}
\leq
\frac{3}{4}
\sum\limits_{(v_i)\in \{0,1\}^n\setminus S_{\lambda'}(\sqrt{n})}
p^{\sum_i v_i}(1-p)^{n-\sum_i v_i}.
\end{equation}
Denote by $\Event_2$ the event that
\eqref{eq: aux 2074450573} holds
(observe that the event is measurable with respect to $\widetilde y$).
Note that for any realization $\widetilde y^0$ of $\widetilde y$ from the event $\Event_2$, we have
\begin{align*}
\sum\limits_{(v_i)\in \{0,1\}^n}
p^{\sum_i v_i}(1-p)^{n-\sum_i v_i}{\bf 1}_{\{
|\sum_{i=1}^n v_i \widetilde y_i^0-\lambda'|\leq 2\sqrt{n}
\}}
&\geq \frac{1}{4}\sum\limits_{(v_i)\in \{0,1\}^n\setminus S_{\lambda'}(\sqrt{n})}
p^{\sum_i v_i}(1-p)^{n-\sum_i v_i}\\
&=\frac{1}{4}\cf\Big(\sum_{i=1}^n b_i y_i,\sqrt{n}\Big).
\end{align*}
This immediately implies
$$\cf\Big(\sum_{i=1}^n b_i \widetilde y_i^0,\sqrt{n}\Big)\geq \frac{1}{8}
\cf\Big(\sum_{i=1}^n b_i y_i,\sqrt{n}\Big).$$

As the last step of the proof, we note that since the variance of the sum $\sum_{i=1}^n (y_i-\widetilde y_i)$
is at most $n/4$, there is an event $\Event_3$ measurable with respect to $\widetilde y$ and of probability at
least $37/48$ such that everywhere on $\Event_3$, $\big|\sum_{i=1}^n (y_i-\widetilde y_i)\big|\leq
\sqrt{12n/11}$.

Finally, since $3-\Prob(\Event_1)-\Prob(\Event_2)-\Prob(\Event_3)<1$,
there exists a realization $y'$ of the random vector $\widetilde y$ from the intersection $\Event_1\cap\Event_2\cap \Event_3$.
It is straightforward to check that $y'$ satisfies all conditions of the lemma.
\end{proof}

\medskip

Given any $p\in(0,1/2]$, $s\in[-1,0]$, any $x\in S^{n-1}$ and $L\geq 1$, we construct integer vector $\spv(p,x,L,s)\in\Z^n$ as follows:
take $y=(y_1,\dots,y_n):=\frac{\sqrt{n}}{\thres_p(x,L)}\,x$ and observe that, by the definition of the threshold,
$$\Prob\Big\{\Big|\sum_{i=1}^n b_i y_i+s\sum_{i=1}^n y_i
\Big|\leq t\Big\}\leq \frac{L\,\thres_p(x,L)}{\sqrt{n}}\,t\;\;\mbox{ for all }t\geq\sqrt{n}.$$
Hence, by Lemma~\ref{l: magic vector}, there is a vector $\spv(p,x,L,s)\in\Z^n$ satisfying
\begin{itemize}\label{p: prop yxl}
\item $\big\|\frac{\sqrt{n}}{\thres_p(x,L)}\,x-\spv(p,x,L,s)\big\|_\infty%=\|y-\spv(p,x,L,s)\|_\infty
\leq 1$;
\item
$\Prob\big\{\big|\sum_{i=1}^n b_i \spv_i(p,x,L,s)+\frac{s\sqrt{n}}{\thres_p(x,L)}\sum_{i=1}^n x_i\big|\leq t\big\}$
$%=\Prob\big\{\big|\sum_{i=1}^n b_i \spv_i(p,x,L,s)+s\sum_{i=1}^n y_i\big|\leq t\big\}
\leq \frac{C_{\text{\tiny\ref{l: magic vector}}}\,L\,\thres_p(x,L)}{\sqrt{n}}\, t$ for all $t\geq \sqrt{n}$;
\item $\cf\big(\sum_{i=1}^n b_i \spv_i(p,x,L,s),\sqrt{n}\big)
\geq %c_{\text{\tiny\ref{l: magic vector}}}\,\cf\big(\sum_{i=1}^n b_i y_i,\sqrt{n}\big)
%=c_{\text{\tiny\ref{l: magic vector}}}\,\cf\big(\sum_{i=1}^n b_i x_i,\thres_p(x,L)\big)=
c_{\text{\tiny\ref{l: magic vector}}}\,L\,\thres_p(x,L)$;
\item $\big|\frac{\sqrt{n}}{\thres_p(x,L)}\sum_{i=1}^n x_i-\sum_{i=1}^n \spv_i(p,x,L,s)\big|
%=\big|\sum_{i=1}^n y_i-\sum_{i=1}^n \spv_i(p,x,L,s)\big|
\leq C_{\text{\tiny\ref{l: magic vector}}} \sqrt{n}$.
\end{itemize}
The vector with the above properties does not have to be unique,
however, from now on we fix a single admissible vector for each $4$--tuple $(p,x,L,s)$.

\medskip

\begin{lemma}\label{l: special permutations}
For any $n\geq 2$ there is a subset $\bf\Pi$ of permutations on $[n]$ with
$|{\bf\Pi}|\leq C_{\text{\tiny\ref{l: special permutations}}}^n$, having the following property.
Let $p\in(0,1/2]$, $\delta\in(0,1/2]$, $s\in[-1,0]$, $\nu\in(0,1]$, $L\geq 1$, and let $x\in\incomp_n(\delta,\nu)$.
Then there is $\sigma=\sigma(x)\in\bf\Pi$ such that the vector
$\widetilde y=\big(\spv_{\sigma(i)}(p,x,L,s)\big)_{i=1}^n$ satisfies
$$|\widetilde y_i|> \frac{\nu}{\thres_p(x,L)}-1\quad\mbox{ for all }i\leq\delta n,$$
and
$$|\widetilde y_i|\leq
\frac{2^{(j+1)/2}}{\sqrt{\delta}\,\thres_p(x,L)}+1,\quad  i> 2^{-j}\delta n,\quad 0\leq j\leq \log_2(\delta n).$$
Here, $C_{\text{\tiny\ref{l: special permutations}}}>0$ is a universal constant.
\end{lemma}
\begin{proof}
If $\delta n<1$ then the statement is empty, and $\bf\Pi$ can be chosen arbitrarily.
We will therefore assume that $\delta n\geq 1$.
We start by defining the collection of permutations $\bf\Pi$.
Let $j_0\geq 0$ be the largest integer such that $\delta n\geq 2^{j_0}$.
For every collection of subsets $[n]\supset I_0\supset\dots\supset I_{j_0}$
with $|I_j|=\lfloor 2^{-j}\delta n\rfloor$, $j=0,\dots,j_0$,
take any permutation $\sigma$ such that
$\sigma\big(\big[\lfloor 2^{-j}\delta n\rfloor\big]\big)=I_j$, $j=0,\dots,j_0$.
We then compose $\bf\Pi$ of all such permutations
(where we pick a single admissible permutation for every collection of subsets).
It is not difficult to check that the total number of admissible collections $[n]\supset I_0\supset\dots\supset I_{j_0}$,
hence the cardinality of $\bf\Pi$, is bounded above by $C^n$
for a universal constant $C>0$.

It remains to check the properties of $\bf\Pi$.
Take any vector $x\in \incomp_n(\delta,\nu)$, and let $[n]\supset I_0(x)\supset\dots\supset I_{j_0}(x)$
be sets of indices corresponding to largest (by absolute value) coordinates of $x$.
Namely, $I_j(x)$ is a subset of cardinality $\lfloor 2^{-j}\delta n\rfloor$ such that
$|x_i|\geq |x_\ell|$ for all $i\in I_j(x)$ and $\ell\in [n]\setminus I_j(x)$.
Let $\sigma\in\bf\Pi$ be a permutation such that
$$\sigma\big(\big[\lfloor 2^{-j}\delta n\rfloor\big]\big)=I_j(x), \quad j=0,\dots,j_0.$$
Set $\widetilde y:=\big(\spv_{\sigma(i)}(p,x,L,s)\big)_{i=1}^n$.

By our construction, $|x_{\sigma(i)}|\geq |x_{\sigma(\ell)}|$ for all $i\leq \delta n<\ell$.
Since $x$ is incompressible,
$$\sum\limits_{\ell>\delta n}x_{\sigma(\ell)}^2\geq\nu^2,$$
whence there exists an index $\ell>\delta n$ such that $|x_{\sigma(\ell)}|> \nu/\sqrt{n}$.
Thus, $|x_{\sigma(i)}|> \nu/\sqrt{n}$ for all $i\leq \delta n$, whence, in view of the definition of vector $\widetilde y$,
$$|\widetilde y_i|> \frac{\nu}{\thres_p(x,L)}-1\quad\mbox{ for all }i\leq\delta n.$$

The upper bounds on coordinates $\widetilde y_i$ are obtained in a similar fashion.
Take any $j\in\{0,\dots,j_0\}$.
Since $|x_{\sigma(i)}|\leq |x_{\sigma(\ell)}|$ for all $\ell\leq 2^{-j}\delta n<i$, and $x$ has Euclidean norm one, we get
$$|x_{\sigma(i)}|\leq \frac{1}{\sqrt{\lfloor 2^{-j}\delta n\rfloor}},\quad i> 2^{-j}\delta n.$$
Hence,
$$|\widetilde y_i|\leq \frac{1}{\sqrt{\lfloor 2^{-j}\delta n\rfloor}}\frac{\sqrt{n}}{\thres_p(x,L)}+1\leq
\frac{2^{(j+1)/2}}{\sqrt{\delta}\,\thres_p(x,L)}+1,\quad  i> 2^{-j}\delta n.$$
\end{proof}

Let $n\geq 2$, $\delta\in[1/n,1/2]$ and $\nu\in(0,1]$. Further, let $T\in(0,1]$ be a number such that
$$\frac{\nu}{T}\geq 2.$$
Define a subset $\mathcal A(n,\delta,\nu,T)\subset\Z^n$ as follows:
we take $\mathcal A(n,\delta,\nu,T)=A_1\times A_2\times\dots\times A_n$,
where
\begin{itemize}

\item For all $1\leq j\leq \log_2(\delta n)$ and $2^{-j}\delta n<i\leq 2^{-j+1}\delta n$, we have
$$A_i:=\Z\cap\,\Big[-\Big\lceil\mbox{$\frac{2^{(j+3)/2}}{\sqrt{\delta}\,T}$}\Big\rceil-1,
\Big\lceil\mbox{$\frac{2^{(j+3)/2}}{\sqrt{\delta}\,T}$}\Big\rceil+1\Big]\setminus
\Big[1-\Big\lfloor\frac{\nu}{T}\Big\rfloor,\Big\lfloor\frac{\nu}{T}\Big\rfloor-1\Big];$$

\item For $i>\delta n$, we have
$$A_i:=\Z\cap 
\,\Big[-\Big\lceil\mbox{$\frac{\sqrt{8}}{\sqrt{\delta}\,T}$}\Big\rceil-1,
\Big\lceil\mbox{$\frac{\sqrt{8}}{\sqrt{\delta}\,T}$}\Big\rceil+1\Big];
$$

\item $A_1:=\Z\cap\,\Big[-\Big\lceil\frac{2\sqrt{n}}{T}\Big\rceil-1,
\Big\lceil\frac{2\sqrt{n}}{T}\Big\rceil+1\Big]\setminus
\Big[1-\Big\lfloor\frac{\nu}{T}\Big\rfloor,\Big\lfloor\frac{\nu}{T}\Big\rfloor-1\Big]$.

\end{itemize}

Lemma~\ref{l: special permutations} immediately implies
\begin{cor}\label{cor: permutations}
For any $n\geq 2$ there is a subset $\bf\Pi$ of permutations on $[n]$ with
$|{\bf\Pi}|\leq C_{\text{\tiny\ref{l: special permutations}}}^n$, having the following property.
Let $p\in(0,1/2]$, $\delta\in[1/n,1/2]$, $s\in[-1,0]$, $\nu\in(0,1]$, $L\geq 1$, $T>0$,
and let $x\in\incomp_n(\delta,\nu)$ be such that $T/2\leq \thres_p(x,L)\leq T$.
Then there is $\sigma=\sigma(x)\in\bf\Pi$ such that the vector
$\big(\spv_{\sigma(i)}(p,x,L,s)\big)_{i=1}^n$ belongs to $\mathcal A(n,\delta,\nu,T)$.
\end{cor}

The next crucial observation, which will enable us to apply results from Section~\ref{s: averaging},
is
\begin{lemma}\label{l: admissibility of A}
For any $\delta\in(0,1/2]$, $\nu\in(0,1]$ there are
$n_{\text{\tiny\ref{l: admissibility of A}}}=n_{\text{\tiny\ref{l: admissibility of A}}}(\delta,\nu)\geq 1$ and
$K_{\text{\tiny\ref{l: admissibility of A}}}=K_{\text{\tiny\ref{l: admissibility of A}}}(\delta,\nu)\geq 1$
with the following property. Take any $n\geq n_{\text{\tiny\ref{l: admissibility of A}}}$,
$T\in(0,\nu/2]$ and set $N:=\big\lfloor\frac{\nu}{T}\big\rfloor-1$. Then
the subset $\mathcal A(n,\delta,\nu,T)$ defined above is $(N,n,K_{\text{\tiny\ref{l: admissibility of A}}},\delta)$--admissible
(with the notion taken from Section~\ref{s: averaging}).
\end{lemma}
%\begin{proof}
%With $\mathcal A(n,\delta,\nu,T)=A_1\times A_2\times\dots\times A_n$, we have
%\begin{itemize}
%\item Each $A_i$ is origin-symmetric;
%\item For every $i>\delta n$, we have $|A_i|= 2\big\lceil\mbox{$\frac{\sqrt{8}}{\sqrt{\delta}\,T}$}\big\rceil+3\geq 2N+1$;
%\item $A_i\cap [-N,N]=\emptyset$ and $\max A_i\geq 2N$ for all $i\leq \delta n$;
%\item $\max A_i< nN$ for all $i$ and all sufficiently large $n$.
%\end{itemize}
%\end{proof}

Now, everything is ready to prove the main result of the paper.
\begin{proof}[Proof of Theorem~A]
Fix any $p\in(0,1/2]$, $\varepsilon\in(0,p/2]$, and assume that $n\geq n_{\text{\tiny\ref{l: compress}}}(\varepsilon,p)$
and $\sqrt{n}\geq 2K_{\text{\tiny\ref{l: threshold}}}/\nu_{\text{\tiny\ref{l: compress}}}(\varepsilon,p)$
(we will impose additional restrictions on $n$ as the proof goes on).
Fix any $s\in[-1,0]$.
Our goal is to estimate from above
$$
\Prob\big\{s_{\min}(B_n(p)+s\,1_n1_n^\top)\leq t/\sqrt{n}\big\},
$$
for any $t>0$.
Set
$$\delta:=\delta_{\text{\tiny\ref{l: compress}}}(\varepsilon,p),\quad \nu:=
\nu_{\text{\tiny\ref{l: compress}}}(\varepsilon,p);\quad \gamma:=\gamma_{\text{\tiny\ref{l: compress}}}(\varepsilon,p).$$
Applying formula \eqref{eq: comp-incomp} and Proposition~\ref{l: compress}, we get for any
$t\leq \gamma n$:
$$\Prob\big\{s_{\min}(B_n(p)+s\,1_n1_n^\top)\leq t/\sqrt{n}\big\}
\leq
\big(1-p+\varepsilon\big)^{n}
+\frac{1}{\delta}\Prob\big\{|\langle\Col_n(B_n(p)+s\,1_n1_n^\top),Y_n\rangle|\leq t/\nu\big\},
$$
where $Y_n$ is a unit random vector measurable with respect to $\Col_1(B_n(p)),\dots,\Col_{n-1}(B_n(p))$
and orthogonal to $\spn\{\Col_1(B_n(p)+s\,1_n1_n^\top),\dots,\Col_{n-1}(B_n(p)+s\,1_n1_n^\top)\}$.
Applying Proposition~\ref{l: compress} the second time, we obtain that the event
$\big\{Y_n\in\comp_n(\delta,\nu)\big\}$
has probability at most $\big(1-p+\varepsilon\big)^{n}$.
Further, for every vector $x\in \incomp_n(\delta,\nu)$, according to Lemma~\ref{l: threshold},
$\thres_p(x,L)\leq \frac{K_{\text{\tiny\ref{l: threshold}}}}{\sqrt{n}}$ whenever $L\geq L_{\text{\tiny\ref{l: threshold}}}$.
Set
$$L:=\max\bigg(L_{\text{\tiny\ref{l: threshold}}},\frac{4L_{\text{\tiny\ref{cor: anticoncentration}}}}{c_{\text{\tiny\ref{l: magic vector}}}\nu}
\bigg).$$
Then, in view of the above, we have
\begin{align*}
\Prob\big\{s_{\min}(B_n(p)+s\,1_n1_n^\top)\leq t/\sqrt{n}\big\}
&\leq
\big(1+\delta^{-1}\big)\big(1-p+\varepsilon\big)^{n}\\
+\frac{1}{\delta}
\sum\limits_{j=0}^\infty
\Prob\Big\{&\mbox{$Y_n\in \incomp_n(\delta,\nu)$ and $|\langle \Col_n(B_n(p)+s\,1_n1_n^\top),Y_n\rangle|\leq t/\nu$}\\
&\mbox{and $\frac{2^{-j-1}K_{\text{\tiny\ref{l: threshold}}}}{\sqrt{n}}<
\thres_p(Y_n,L)\leq \frac{2^{-j} K_{\text{\tiny\ref{l: threshold}}}}{\sqrt{n}}$}\Big\}.
\end{align*}
Further, for any $j\geq 0$,
using the independence of $Y_n$ and $\Col_n(B_n(p)+s\,1_n1_n^\top)$ and the definition of the threshold,
we can write
\begin{align*}
\Prob\Big\{&\mbox{$|\langle \Col_n(B_n(p)+s\,1_n1_n^\top),Y_n\rangle|\leq t/\nu$
and $\frac{2^{-j-1}K_{\text{\tiny\ref{l: threshold}}}}{\sqrt{n}}<
\thres_p(Y_n,L)\leq \frac{2^{-j} K_{\text{\tiny\ref{l: threshold}}}}{\sqrt{n}}$}\Big\}\\
&\hspace{3cm}\leq \mbox{$L\,\max\big(\frac{2^{-j} K_{\text{\tiny\ref{l: threshold}}}}{\sqrt{n}},\frac{t}{\nu}\big)$}
\,\Prob\Big\{\mbox{$\frac{2^{-j-1}K_{\text{\tiny\ref{l: threshold}}}}{\sqrt{n}}<
\thres_p(Y_n,L)\leq \frac{2^{-j} K_{\text{\tiny\ref{l: threshold}}}}{\sqrt{n}}$}\Big\}.
\end{align*}
Hence, for every $t\leq\gamma n$,
\begin{align*}
\Prob\big\{&s_{\min}(B_n(p)+s\,1_n1_n^\top)\leq t/\sqrt{n}\big\}\\
&\leq
\big(1+\delta^{-1}\big)\big(1-p+\varepsilon\big)^{n}
+\frac{L}{\delta}\max\bigg(\frac{(1-p+\varepsilon)^{n} K_{\text{\tiny\ref{l: threshold}}}}{\sqrt{n}},\frac{t}{\nu}\bigg)\\
&\hspace{1cm}+\frac{1}{\delta}
\sum\limits_{j=0}^{\lfloor -n\,\log_2(1-p+\varepsilon)\rfloor}
\Prob\Big\{\mbox{$Y_n\in \incomp_n(\delta,\nu)$ and $\frac{2^{-j-1}K_{\text{\tiny\ref{l: threshold}}}}{\sqrt{n}}<
\thres_p(Y_n,L)\leq \frac{2^{-j} K_{\text{\tiny\ref{l: threshold}}}}{\sqrt{n}}$}\Big\}.
\end{align*}
Fix any $j\in\{0,1,\dots,\lfloor -n\,\log_2(1-p+\varepsilon)\rfloor\}$ and
set $T:=\frac{2^{-j} K_{\text{\tiny\ref{l: threshold}}}}{\sqrt{n}}$ and
$$\mbox{$N:=\Big\lfloor\frac{\nu}{T}\Big\rfloor-1;\;
\mathcal A:=\mathcal A(n,\delta,\nu,T);\;
M:=\log\big(8 (C+C_{\text{\tiny\ref{l: magic vector}}})C_{\text{\tiny\ref{l: tensorization}}}
C_{\text{\tiny\ref{l: magic vector}}}(1+C_{\text{\tiny\ref{l: magic vector}}})C_{\text{\tiny\ref{l: special permutations}}}\,
L_{\text{\tiny\ref{l: admissibility of A}}}L \nu\big)$},$$
where $C>0$ denotes the constant such that
$$\Prob\big\{\|B_n^1(p)-p\,1_{n-1}1_n^\top\|\geq C\sqrt{n}\big\}\leq 2^{-n}$$
(which exists, according to Lemma~\ref{l: sp norm}).
Further, let ${\bf\Pi}$ be the set of permutations from Corollary~\ref{cor: permutations}.
Take any $x\in\incomp_n(\delta,\nu)$ such that
$T/2<\thres_p(x,L)\leq T$.
Then the vector $\spv(p,x,L,s)$ satisfies (see page~\pageref{p: prop yxl})
\begin{itemize}
\item[(a)] $\big\|\frac{\sqrt{n}}{\thres_p(x,L)}\,x-\spv(p,x,L,s)\big\|_\infty\leq 1$;
\item[(b)]
$\Prob\big\{\big|\sum_{i=1}^n b_i \,\spv_i(p,x,L,s)+s\,\frac{\sqrt{n}}{\thres_p(x,L)}\sum_{i=1}^n x_i\big|\leq \tau\big\}
\leq \frac{C_{\text{\tiny\ref{l: magic vector}}}\,L\,T}{\sqrt{n}}\, \tau$ for all $\tau\geq \sqrt{n}$;
\item[(c)] $\cf\big(\sum_{i=1}^n b_i \,\spv_i(p,x,L,s),\sqrt{n}\big)
\geq c_{\text{\tiny\ref{l: magic vector}}}\,L\,\thres_p(x,L)     %\cf\big(\sum_{i=1}^n b_i x_i,\thres_p(x,L)\big)
\geq \frac{c_{\text{\tiny\ref{l: magic vector}}}}{2}LT\geq \frac{c_{\text{\tiny\ref{l: magic vector}}}L\nu}{4N}$;
\item[(d)] $\big|\sum_{i=1}^n \frac{\sqrt{n}}{\thres_p(x,L)}\,x_i-\sum_{i=1}^n \spv_i(p,x,L,s)\big|
\leq C_{\text{\tiny\ref{l: magic vector}}}\sqrt{n}$.
\end{itemize}
Note that a combination of (b) and (d) gives
$$
\Prob\Big\{\Big|\sum_{i=1}^n b_i \,\spv_i(p,x,L,s)+s\,\sum_{i=1}^n \spv_i(p,x,L,s)\Big|\leq \tau\Big\}
\leq \frac{C_{\text{\tiny\ref{l: magic vector}}}(1+C_{\text{\tiny\ref{l: magic vector}}})\,L\,T}{\sqrt{n}}\, \tau\;\;\;
\mbox{ for all $\tau\geq \sqrt{n}$}.
$$
Define the subset $D\subset\mathcal A$ as
\begin{align*}
D:=\Big\{&y\in \mathcal A:\;\cf\Big(\sum_{i=1}^n b_i y_i,\sqrt{n}\Big)
\geq \frac{c_{\text{\tiny\ref{l: magic vector}}}L\nu}{4N}\;\;\mbox{ and}\\
&\Prob\Big\{\Big|\sum_{i=1}^n b_i y_i+s\,\sum_{i=1}^n y_i\Big|\leq \tau\Big\}
\leq \frac{C_{\text{\tiny\ref{l: magic vector}}}(1+C_{\text{\tiny\ref{l: magic vector}}})\,L\,T}{\sqrt{n}}\,
\tau\;\;\mbox{ for all }\;\;\tau\geq \sqrt{n}
\Big\},
\end{align*}
and let $\Net_T$ be defined as
$$\Net_T:=\big\{y\in\Z^n:\;(y_{\sigma(i)})_{i=1}^n\in D\mbox{ for some $\sigma\in{\bf\Pi}$}\big\}.$$
Then, by Corollary~\ref{cor: permutations} and the above remarks,
$\spv(p,x,L,s)\in\Net_T$ for every $x\in\incomp_n(\delta,\nu)$ with
$T/2<\thres_p(x,L)\leq T$.
Set $Q:=\big\{z\in\R^n:\;\big|\sum_{i=1}^n z_i\big|\leq C_{\text{\tiny\ref{l: magic vector}}}\sqrt{n}\big\}$.
Then the last assertion, together with properties (a) and (d) above, implies
\begin{align*}
\Big\{\frac{\sqrt{n}}{\thres_p(x,L)}x:\;
x\in\incomp_n(\delta,\nu),\;\;T/2<\thres_p(x,L)\leq T\Big\}
\subset \Net_T+[-1,1]^n\cap Q.
\end{align*}
Thus, we obtain the relation
\begin{align*}
\Prob\Big\{&\mbox{$Y_n\in \incomp_n(\delta,\nu)$ and $\frac{2^{-j-1}K_{\text{\tiny\ref{l: threshold}}}}{\sqrt{n}}<
\thres_p(Y_n,L)\leq \frac{2^{-j} K_{\text{\tiny\ref{l: threshold}}}}{\sqrt{n}}$}\Big\}\\
&\leq \Prob\big\{\mbox{There exists
$y\in \Net_T+[-1,1]^n\cap Q$ such that $(B_n^1(p)+s\,1_{n-1}1_n^\top) y=0$}\big\}.
\end{align*}

Now, let us estimate the probability that $\|(B_n^1(p)+s\,1_{n-1}1_n^\top) y\|_2$ is small for a fixed $y\in\Net_T$.
By our definition of the set $\Net_T$, we have
$$\Prob\big\{\big|\langle \Row_k(B_n^1(p)+s\,1_{n-1}1_n^\top), y\rangle\big|\leq \tau\big\}
\leq \mbox{$\frac{C_{\text{\tiny\ref{l: magic vector}}}(1+C_{\text{\tiny\ref{l: magic vector}}})
\,L\,T}{\sqrt{n}}$}\, \tau\;\;\mbox{ for all }\;\;\tau\geq \sqrt{n}\;\;
\mbox{ and }\;\;k\leq n-1.$$
Hence, appying Lemma~\ref{l: tensorization}, we get
$$
\Prob\big\{\|(B_n^1(p)+s\,1_{n-1}1_n^\top) y\|_2\leq \tau\sqrt{n-1}\big\}\leq \big(C_{\text{\tiny\ref{l: tensorization}}}\,
\mbox{$\frac{C_{\text{\tiny\ref{l: magic vector}}}(1+C_{\text{\tiny\ref{l: magic vector}}})
\,L\,T}{\sqrt{n}}$}\,\tau\big)^{n-1}\mbox{ for all }\tau\geq \sqrt{n}.
$$
Observe that for any $z\in [-1,1]^n\cap Q$ we have
\begin{align*}
\|(B_n^1(p)+s\,1_{n-1}1_n^\top) z\|_2
&\leq \|z\|_2\,\|B_n^1(p)-p\,1_{n-1}1_n^\top\|
+|s+p|\,\|1_{n-1}1_n^\top z\|_2\\
&\leq \sqrt{n}\,\|B_n^1(p)-p\,1_{n-1}1_n^\top\|
+C_{\text{\tiny\ref{l: magic vector}}}n,
\end{align*}
where we have used that $s\in[-1,0]$.
Then the above relations, together with a net argument, imply
\begin{align*}
\Prob\big\{&\mbox{There exists $y\in \Net_T+[-1,1]^n\cap Q$ such that $(B_n^1(p)+s\,1_{n-1}1_n^\top) y=0$}\big\}\\
&\leq \Prob\big\{\|B_n^1(p)-p\,1_{n-1}1_n^\top\|\geq C\sqrt{n}\big\}+|\Net_T|\,\max\limits_{y\in\Net_T}
\Prob\big\{\|(B_n^1(p)+s\,1_{n-1}1_n^\top) y\|_2\leq Cn+C_{\text{\tiny\ref{l: magic vector}}}n\big\}\\
&\leq 2^{-n}+|\Net_T|\,
\big(2 (C+C_{\text{\tiny\ref{l: magic vector}}})C_{\text{\tiny\ref{l: tensorization}}}
C_{\text{\tiny\ref{l: magic vector}}}(1+C_{\text{\tiny\ref{l: magic vector}}})\,L\,T\big)^{n-1}.
\end{align*}

The last --- and the most important --- step of the proof is to bound from above the cardinality of $\Net_T$.
In view of Corollary~\ref{cor: permutations} and the definition of $D$ and $\Net_T$, we have
$$|\Net_T|\leq C_{\text{\tiny\ref{l: special permutations}}}^n|D|.$$
Further, observe that by Lemma~\ref{l: admissibility of A}, the set $\mathcal A$
is $(N,n,K_{\text{\tiny\ref{l: admissibility of A}}},\delta)$--admissible.
Hence, Corollary~\ref{cor: anticoncentration} is applicable, and the definition of $D$
gives for all $n$ large enough:
$$|D|\leq e^{-Mn}|\mathcal A|\leq e^{-Mn}(K_{\text{\tiny\ref{l: admissibility of A}}} N)^n.$$
Combining this with the above relations and recalling that $N=\big\lfloor \frac{\nu}{T}\big\rfloor-1$, we obtain
\begin{align*}
\Prob\big\{&\mbox{There exists $y\in \Net_T+[-1,1]^n\cap Q$ such that $(B_n^1(p)+s\,1_{n-1}1_n^\top) y=0$}\big\}\\
&\leq 2^{-n}+e^{-Mn}(K_{\text{\tiny\ref{l: admissibility of A}}} N)^n
C_{\text{\tiny\ref{l: special permutations}}}^n
\big(2 (C+C_{\text{\tiny\ref{l: magic vector}}})C_{\text{\tiny\ref{l: tensorization}}}
C_{\text{\tiny\ref{l: magic vector}}}(1+C_{\text{\tiny\ref{l: magic vector}}})\,L\,T\big)^{n-1}\\
&\leq 
2^{-n}+e^{-Mn}(\sqrt{n}2^n/K_{\text{\tiny\ref{l: threshold}}})
\big(2 (C+C_{\text{\tiny\ref{l: magic vector}}})C_{\text{\tiny\ref{l: tensorization}}}
C_{\text{\tiny\ref{l: magic vector}}}(1+C_{\text{\tiny\ref{l: magic vector}}})C_{\text{\tiny\ref{l: special permutations}}}\,
K_{\text{\tiny\ref{l: admissibility of A}}}L \nu\big)^{n}\\
&\leq 2^{-n}+\sqrt{n} 2^{-n}/K_{\text{\tiny\ref{l: threshold}}}
\end{align*}
for all sufficiently large $n$,
where the last relation follows from the choice of $M$.

Returning to the small ball probability for $s_{\min}(B_n(p)+s\,1_n1_n^\top)$, we get
\begin{align*}
\Prob\big\{s_{\min}(B_n(p)+s\,1_n1_n^\top)\leq t/\sqrt{n}\big\}
&\leq
\big(1+\delta^{-1}\big)\big(1-p+\varepsilon\big)^{n}\\
&\hspace{1cm}+\frac{L}{\delta}\max\bigg(\frac{(1-p+\varepsilon)^{n} K_{\text{\tiny\ref{l: threshold}}}}{\sqrt{n}},\frac{t}{\nu}\bigg)\\
&\hspace{1cm}+\frac{n}{\delta}\big(2^{-n}+\sqrt{n} 2^{-n}/K_{\text{\tiny\ref{l: threshold}}}\big)\\
&\leq \big(1-p+2\varepsilon\big)^{n}+C_{\varepsilon,p}\,t
\end{align*}
for all sufficiently large $n$.
Since $\varepsilon\in(0,p/2]$ was chosen arbitrarily, the result follows.
\end{proof}

\noindent
{\bf Acknowledgement.}
I would like to thank the Department of Mathematical and Statistical Sciences, University of Alberta,
which I visited in December 2018 and where the first draft of this work was completed.
I would also like to thank Prof.\ Terence Tao
and the anonymous Referees for valuable remarks.


\begin{thebibliography}{99}

\bibitem{AK}
{
N. Alon\ and\ B. Klartag, Optimal compression of approximate inner products and dimension reduction, in {\it 58th Annual IEEE Symposium on Foundations of Computer Science---FOCS 2017}, 639--650, IEEE Computer Soc., Los Alamitos, CA. MR3734268
}

\bibitem{Arratia}
{
R. Arratia\ and\ S. DeSalvo, On the singularity of random Bernoulli matrices---novel integer partitions and lower bound expansions,
Ann. Comb. {\bf 17} (2013), no.~2, 251--274. MR3056767
}

\bibitem{BVW}
{
J. Bourgain, V. H. Vu\ and\ P. M. Wood, On the singularity probability of discrete random matrices,
J. Funct. Anal. {\bf 258} (2010), no.~2, 559--603. MR2557947
}

\bibitem{CT}
{
D. Chafa\"{i}\ and\ K. Tikhomirov, On the convergence of the extremal eigenvalues of empirical covariance
matrices with dependence, Probab. Theory Related Fields {\bf 170} (2018), no.~3-4, 847--889. MR3773802
}

\bibitem{Erdos LO}
{
P. Erd\"{o}s, On a lemma of Littlewood and Offord, Bull. Amer. Math. Soc. {\bf 51} (1945), 898--902. MR0014608
}

\bibitem{KKS}
{
J. Kahn, J. Koml\'{o}s\ and\ E. Szemer\'{e}di, On the probability that a random $\pm 1$-matrix is singular,
J. Amer. Math. Soc. {\bf 8} (1995), no.~1, 223--240. MR1260107
}

\bibitem{KlLiv}
{
B. Klartag, G. Livshyts,
The lower bound for Koldobsky's slicing inequality via random rounding,
arXiv:1810.06189
}

\bibitem{Komlos}
{
J. Koml\'{o}s, On the determinant of $(0,\,1)$ matrices, Studia Sci. Math. Hungar {\bf 2} (1967), 7--21. MR0221962
}

\bibitem{LO lemma}
{
J. E. Littlewood\ and\ A. C. Offord, On the number of real roots of a random algebraic equation. III,
Rec. Math. [Mat. Sbornik] N.S. {\bf 12(54)} (1943), 277--286. MR0009656
}

\bibitem{LPRT}
{
A. E. Litvak, A. Pajor, M. Rudelson, N. Tomczak-Jaegermann,
Smallest singular value of random matrices and geometry of random polytopes, Adv. Math. {\bf 195} (2005), no.~2, 491--523. MR2146352
}

\bibitem{Livshyts}
{
G. V. Livshyts, The smallest singular value of heavy-tailed not necessarily i.i.d.\ random matrices via random rounding, preprint.
}

\bibitem{RRounding}
{
P. Raghavan\ and\ C. D. Thompson, Randomized rounding: a technique for provably good algorithms and algorithmic proofs,
Combinatorica {\bf 7} (1987), no.~4, 365--374. MR0931194
}

\bibitem{Rogozin}
{
B. A. Rogozin, On the increase of dispersion of sums of independent random variables,
Teor. Verojatnost. i Primenen {\bf 6} (1961), 106--108. MR0131894

Translated into English in {\it Theory Probab. Appl. {\bf 6} (1961), 97--99.}
}

\bibitem{Rudelson ann math}
{
M. Rudelson, Invertibility of random matrices: norm of the inverse, Ann. of Math. (2) {\bf 168} (2008), no.~2, 575--600. MR2434885
}

\bibitem{RV adv}
{
M. Rudelson\ and\ R. Vershynin, The Littlewood-Offord problem and invertibility of random matrices,
Adv. Math. {\bf 218} (2008), no.~2, 600--633. MR2407948
}

\bibitem{RV rect}
{
M. Rudelson\ and\ R. Vershynin, Smallest singular value of a random rectangular matrix,
Comm. Pure Appl. Math. {\bf 62} (2009), no.~12, 1707--1739. MR2569075
}

\bibitem{TV disc1}
{
T. Tao\ and\ V. Vu, On random $\pm1$ matrices: singularity and determinant,
Random Structures Algorithms {\bf 28} (2006), no.~1, 1--23. MR2187480
}

\bibitem{TV disc2}
{
T. Tao\ and\ V. Vu, On the singularity probability of random Bernoulli matrices,
J. Amer. Math. Soc. {\bf 20} (2007), no.~3, 603--628. MR2291914
}

\bibitem{TV circular}
{
T. Tao\ and\ V. Vu, Random matrices: the circular law, Commun. Contemp. Math. {\bf 10} (2008), no.~2, 261--307. MR2409368
}

\bibitem{TV ann math}
{
T. Tao\ and\ V. H. Vu, Inverse Littlewood-Offord theorems and the condition number of random discrete matrices,
Ann. of Math. (2) {\bf 169} (2009), no.~2, 595--632. MR2480613
}

\bibitem{V12}
{
R. Vershynin, Introduction to the non-asymptotic analysis of random matrices, in {\it Compressed sensing},
210--268, Cambridge Univ. Press, Cambridge. MR2963170
}

\bibitem{Vu survey}
{
V. Vu, Random discrete matrices, in {\it Horizons of combinatorics}, 257--280, Bolyai Soc. Math. Stud., 17, Springer, Berlin. MR2432537
}

\bibitem{Vu ICM 2014}
{
V. H. Vu, Combinatorial problems in random matrix theory, in
{\it Proceedings of the International Congress of Mathematicians---Seoul 2014. Vol. IV}, 489--508, Kyung Moon Sa, Seoul. MR3727622
}

\end{thebibliography}
\end{document}